\documentclass[11pt]{article}
\usepackage[margin=1in]{geometry}
\usepackage[utf8]{inputenc}
\usepackage{amsmath}
\usepackage{amssymb}
\usepackage{mathtools}
\usepackage{amsthm}
\usepackage{authblk}
\usepackage{threeparttable}
\usepackage{breakcites}
\usepackage{natbib}
\usepackage{graphicx}
\usepackage{xcolor}
\usepackage[title]{appendix}
\newtheorem{theorem}{Theorem}
\newtheorem{proposition}{Proposition}
\newtheorem{lemma}{Lemma}
\newtheorem{corollary}{Corollary}
\theoremstyle{remark}
\newtheorem{remark}{Remark}
\newtheorem{assumption}{Assumption}
\theoremstyle{definition}
\newtheorem{example}{Example}

\newenvironment{assumptionp}[1]
  {\renewcommand{\theassumption}{\ref{#1}$'$}%
   \addtocounter{assumption}{-1}%
   \begin{assumption}}
  {\end{assumption}}

\begin{document}

\title{Kernel Two-Sample Tests in High Dimension: Interplay Between Moment Discrepancy and Dimension-and-Sample Orders}
\author{Jian Yan}
\author{Xianyang Zhang\thanks{Zhang and Yan acknowledge partial support from NSF DMS-1607320 and NSF DMS-1811747. Address correspondence to Xianyang Zhang
(zhangxiany@stat.tamu.edu).}}
\date{}
\affil{Texas A\&M University}

\maketitle

\textbf{Abstract} Motivated by the increasing use of kernel-based metrics for high-dimensional and large-scale data, we study the asymptotic behavior of kernel two-sample tests when the dimension and sample sizes both diverge to infinity. We focus on the maximum mean discrepancy (MMD) using isotropic kernel, including MMD with the Gaussian kernel and the Laplace kernel, and the energy distance as special cases. We derive asymptotic expansions of the kernel two-sample statistics, based on which we establish the central limit theorem (CLT) under both the null hypothesis and the local and fixed alternatives. The new non-null CLT results allow us to perform asymptotic exact power analysis, which reveals a delicate interplay between the moment discrepancy that can be detected by the kernel two-sample tests and the dimension-and-sample orders. The asymptotic theory is further corroborated through numerical studies. 
\\
\strut \textbf{Keywords:} 
High dimensionality; Kernel method; Power analysis; Two-sample testing.

\section{Introduction}
Nonparametric two-sample testing, aiming to determine whether two collections of samples are from the same distribution without specifying the exact parametric forms of the distributions, is one of the fundamental problems in statistics. The most traditional tools in this domain include the Kolmogorov-Smirnov test \citep{smirnov1939estimation}, Cramér–von Mises test \citep{anderson1962distribution}, Wald-Wolfowitz runs test \citep{wald1940test}, and Mann–Whitney U test \citep{mann1947test}. Extensions and generalizations of these tests were studied in \citet{darling1957kolmogorov}, \citet{bickel1969distribution}, \citet{friedman1979multivariate}, among many others.

Modern nonparametric tests have been developed based on integral probability metrics (see \citet{sriperumbudur2012empirical}). Notable members include the kernel-based maximum mean discrepancy (MMD) two-sample test \citep{gretton2012kernel} and energy distance-based two-sample test \citep{szekely2004testing}. These metrics are gaining increasing popularity in both the statistics and machine learning communities and they have been applied to a plethora of statistical problems including goodness-of-fit testing \citep{szekely2005new}, nonparametric analysis of variance \citep{rizzo2010disco}, change-point detection \citep{matteson2014nonparametric,chakraborty2021high}, finding representative points of a distribution \citep{mak2018support}, and controlled variable selection \citep{romano2020deep}.

In the past decade, high-dimensional and large-scale data are becoming prevalent, particularly in machine learning and deep learning applications, and the attention to the kernel two-sample testing for high-dimensional and large-scale data is also naturally increasing. Examples include training and evaluating deep generative networks \citep{li2015generative}, deep transfer learning \citep{long2017deep}, and variational auto-encoder \citep{louizos2015variational}. 

Motivated by the increasing use of kernel and distance-based metrics for high-dimensional and large-scale data, we propose to investigate the behaviors of these metrics when the dimension and sample sizes both diverge to infinity. Through numerical studies, \cite{ramdas2015decreasing} showed that the power of MMD-based two-sample tests drops polynomially with increasing dimension against the alternatives, where the Kullback-Leibler divergence between the pairs of distributions remains constant. When the sample sizes are fixed and the dimension diverges to infinity, \cite{chakraborty2021new} showed that the energy distance and MMD only detect the equality of means and the traces of covariance matrices; see \cite{zhu2021interpoint} for similar findings on the permutation-based tests. Under the null hypothesis, a recent arXiv paper by \cite{gao2021two} obtained the central limit theorem for the studentized sample MMD as both the dimension and sample sizes diverge to infinity. Other related works include \cite{szekely2013distance}, \cite{zhu2020distance}, \cite{gao2021asymptotic} and \cite{han2021generalized} which study the kernel and distance-based dependence metrics in high dimensions.

Despite the recent advances in understanding the behaviors of the kernel and distance-based metrics in high dimension, a complete picture of the power behavior of the kernel two-sample tests under different dimension-and-sample orders is still lacking. To address the issue, we utilize asymptotic expansions of the kernel two-sample statistics, based on which we establish the central limit theorem (CLT) under both the null hypothesis and the local and fixed alternatives. Built on the null CLT and a new variance estimator, we develop a MMD-based two-sample test in high dimension. The new non-null CLT results facilitate the asymptotic exact power analysis, which reveals an interesting interplay between the moment discrepancy that can be detected by the kernel two-sample tests and the dimension-and-sample orders. Specifically, the findings are as follows. 
First, when $N$ grows slower than $\surd{p}$, MMD detects the discrepancy lying on the means and the traces of covariance matrices. When $N$ grows slower than $p^{2r-{\color{red}1/2}}$ for $r\ge {\color{red}1}$, MMD mainly detects the discrepancy lying on the first ${\color{red}2r}$ moments. Table \ref{tb1} summarizes the features and explicit quantities captured by MMD for different dimension-and-sample orders. The features characterized by MMD are encoded in the quantities $\Delta_0$ and $T_s\ (s\ge 1)$. In particular, $\Delta_0$ captures the discrepancy between the means and the traces of covariance matrices, while $T_s$ quantifies the differences between the tensors formed by higher-order moments, see Section \ref{sec3.4}. 
Second, we summarize several scenarios in Tables \ref{tb2}--\ref{tb3}, where the MMD-based test will have trivial or non-trivial power depending on the moment discrepancy and the dimension-and-sample orders. 
Third, as a byproduct, we also discover the impact of the kernel and bandwidth on the asymptotic power; see Section \ref{sec:impact-kernel} for a detailed discussion. 
Lastly, our asymptotic theory is applicable to a large class of kernels of the type $k(x,y)=f(\|x-y\|_{2}^2/\gamma)$ for $x,y\in\mathbb{R}^p$ and the bandwidth $\gamma>0$, where $f$ is a sufficiently smooth function. Examples include the MMD with the Gaussian kernel, Laplace kernel and rational quadratic kernel, and the energy distance. 

Among the existing literature, the most closely related paper to ours is the one by \citet{gao2021two}. Yet, our results differ significantly from theirs in four crucial aspects: (i) \citet{gao2021two} proved the null CLT under a set of abstract moment conditions, which could be challenging to verify (see also Theorem B.1 of \citet{chakraborty2021new}), while our results are obtained under a general multivariate model specified in (\ref{eq-m}) below and do not impose any restriction on the dimension-and-sample orders for the null CLT; (ii) the non-null CLTs are new and have not been obtained previously in the literature; (iii) \citet{gao2021two} showed that when the dimension grows much slower than the sample sizes, the MMD-based two-sample test has asymptotic power approaching one under some specific alternatives, whereas we study the asymptotic power under a broader range of dimension-and-sample orders. We discuss cases when the asymptotic power is trivial or non-trivial; (iv) the techniques employed in our proofs are also very different from those in \citet{gao2021two}.
Therefore, our findings complement those in the recent literature and shed new light on kernel two-sample tests for high-dimensional and large-scale data.

\begin{table}
\centering
\footnotesize
\def~{\hphantom{0}}
\caption{Dimension and sample size orders, main features captured by MMD as well as the explicit quantities affecting the asymptotic power. }{%
\begin{tabular}{lll}
%\\
\\
Dimension and sample size orders & Main features captured & Explicit quantities\\
$N=o(\surd{p})$ &  mean and trace of covariance & $\Delta_0$ \\
$N=o(p^{3/2})$  & mean and covariance & $\Delta_0,T_1$\\
$N=o(p^{2r-{\color{red}1/2}})$ & the first ${\color{red}2r}$ moments & $\Delta_0,T_1,\ldots,T_{{\color{red}2r-1}}$ \\
fixed $p$, growing $N$ & total homogeneity & $\text{MMD}^2(P_X,P_Y)$\\ 
\end{tabular}}
\label{tb1}
\begin{tablenotes}
\item The quantities $\Delta_0$ and $T_s\ (s=1,\ldots,l-1)$ are defined in (\ref{delta0-2}), (\ref{T1-2}) and (\ref{eq-Ts}) respectively.
\end{tablenotes}
\end{table}

The take-home message of this work is that the faster the sample sizes grow relative to the dimension, the higher-order moment discrepancy the kernel tests are capable of detecting. Or, put differently, the faster the dimension grows relative to the sample sizes, the stronger linearization effect the kernel and distance-based metrics exhibit. In a broader sense, our discovery is reminiscent of the phenomenon that nonlinear kernel gets linearized as the dimension grows with the sample size \citep{el2010spectrum}. We expect similar phenomenon to hold for other kernel methods in the high-dimensional setting. 

To end the introduction, we introduce some notation that will be used throughout the paper. For two real-valued sequences $\{x_p\}$ and $\{a_p\}$, we write $x_{p}=O(a_{p})$ if and only if $|x_{p}/a_{p}|\leq C$ for $p\ge p_{0}$ and some constant $C>0$; $x_{p}=o(a_{p})$ if and only if $\lim_{p\rightarrow\infty}|x_{p}/a_{p}|=0$. Denote by $x_{p}=\Omega(a_{p})$ if and only if $|x_{p}/a_{p}|\ge C>0$ for $p\ge p_{0}$ and some constant $C$; $x_{p}=\omega(a_{p})$ if and only if $\lim_{p\rightarrow\infty}|x_{p}/a_{p}|=\infty$; $x_{p}=\Theta(a_{p})$ if and only if $0<C_{2}\le|x_{p}/a_{p}|\leq C_{1}$ for $p\ge p_{0}$ and some constants $C_{1},C_{2}>0$. For a real-valued function $f$ that is $L$-times differentiable, we denote its $l$th derivative at $x$ by $f^{(l)}(x)$ for $l=1,\ldots,L$. The function $f$ is said to be of (differentiability) class $C^L$ if the derivatives $\{f^{(l)}:l=1,\ldots,L\}$ exist and are continuous. For a matrix or tensor $\mathcal{T}$, we let $\|\mathcal{T}\|_{\rm op}$ be its operator norm, and $\|\mathcal{T}\|_{\rm F}$ be its Frobenius norm.

\section{Preliminaries}
Let $X,Y\in \mathbb{R}^p$ be two random vectors, with respective probability measures $P_X$ and $P_Y$. Given the observations $\{X_i\}^{n}_{i=1}$ and $\{Y_j\}^{m}_{j=1}$ independently and identically distributed (i.i.d.) from $P_X$ and $P_Y$, respectively, we aim to test the following hypothesis
\[
H_0:P_X=P_Y \quad \text{versus} \quad H_a:P_X\neq P_Y.   
\]

A broad class of metrics for quantifying the discrepancy between $P_X$ and $P_Y$ is the integral probability metric (IPM) defined as
\[
\mathcal{M}(P_X,P_Y)=\sup_{f\in\mathcal{F}}\left|E\{f(X)\}-E\{f(Y)\}\right|,    
\]
where $\mathcal{F}$ is a class of real-valued bounded measurable functions on $\mathbb{R}^p$ \citep{sriperumbudur2012empirical}. Notable examples include the Dudley metric, the Wasserstein distance, the total variation distance and maximum mean discrepancy (MMD). When $\mathcal{F}$ is a unit ball in the reproducing kernel Hilbert space $\mathcal{H}_k$ with associated kernel $k(\cdot,\cdot)$, the IPM becomes the MMD which admits a closed-form given by \citep{gretton2012kernel}
\[
\text{MMD}^2(P_X,P_Y)=\|\mu_{P_{X}}-\mu_{P_{Y}}\|_{\mathcal{H}_k}^2=E\{k(X,X')\}+E\{k(Y,Y')\}-2E\{k(X,Y)\},
\]
where $\mu_{P}$ is the kernel mean embedding defined as $\mu_{P}=\int k(\cdot,x)dP(x)$, and $(X',Y')$ is an i.i.d. copy of $(X,Y)$. The kernel $k$ is said to be characteristic if the mapping $P\rightarrow\mu_{P}$ is injective. If the kernel $k$ is characteristic, then $\text{MMD}(P_X,P_Y)=0$ if and only if $P_X=P_Y$.
An unbiased estimator of MMD \citep{gretton2012kernel} is given by
\begin{equation}\label{eq-mmds}
\text{MMD}_{n,m}^{2}=\frac{1}{n(n-1)}\sum_{i_{1}\neq i_{2}}k(X_{i_{1}},X_{i_{2}})+\frac{1}{m(m-1)}\sum_{j_{1}\neq j_{2}}k(Y_{j_{1}},Y_{j_{2}})-\frac{2}{nm}\sum_{i,j}k(X_{i},Y_{j}),
\end{equation}
which is a two-sample U-statistic. Under $H_{0}$, $\text{MMD}_{n,m}^{2}$ is degenerate, and when $p$ is fixed, 
\[
\frac{nm}{n+m}\text{MMD}_{n,m}^{2}\rightarrow\sum_{l=1}^{\infty}\lambda_{l}(z_{l}^{2}-1)\text{ in distribution, as }n,m\rightarrow\infty,
\]
where $\{z_{l}\}$ is a sequence of independent $N(0,1)$ random variables, and $\lambda_{l}$'s depend on the distribution $P_X=P_Y$ and the kernel $k$. 

Here we consider the isotropic kernel. A kernel is said to be isotropic if it is only a function of distance between its arguments, i.e, $k(x,y)=f(\|x-y\|_{2}^{2}/\gamma)$, where $f$ is a real-valued function on $[0,+\infty)$ and $\gamma$ is a bandwidth parameter. Examples include the widely-used Gaussian kernel $f(x)=\exp(-x)$, the Laplace kernel $f(x)=\exp(-\surd{x})$, and the rational quadratic kernel $f(x)=(1+x)^{-\alpha}$ for $\alpha>0$. 
On top of that, we want to emphasize that our analysis does not require $k(x,y)$ to be even a valid kernel function. Indeed, our arguments are applicable to any two-sample metrics with $k(x,y)=f(\|x-y\|_{2}^{2}/\gamma)$, where $f$ satisfies Assumption \ref{assumption5} below. In particular, up to a scaling factor $1/\surd{\gamma}$, the energy distance corresponds to $f(x)=-\surd{x}$, which can also be viewed as a special case. Besides, one can easily verify that, when $f(x)=-x$, $\text{MMD}_{n,m}^{2}$ reduces to the two-sample statistic for testing the equality of high-dimensional means \citep{chen2010two}, up to a scaling factor $2/\gamma$.

\section{Null and non-null central limit theorems}\label{sec:stationary}
\subsection{Assumptions and asymptotic expansions}\label{sec:expan}
We cope with MMD using the kernel of the type $k_{p}(x,y)=f(\|x-y\|_{2}^{2}/\gamma_{p})$, where $\gamma_p=\Theta(p)$. Suppose $\lim_{p\rightarrow\infty} \gamma_p/p=c_0$, then $k_p(x,y)$ is asymptotically equivalent to $\tilde{f}(\|x-y\|_{2}^{2}/p)$ with $\tilde{f}(\cdot)=f(\cdot/c_0)$. Thus, without loss of generality, we shall assume below that $\gamma_p=p$ and suppress the dependence of $k$ on $p$ when there is no confusion. We remark that the order of the bandwidth chosen by the commonly used median heuristic is indeed $\Theta(p)$ \citep{reddi2015high}.

Throughout the analysis, we shall make the following assumptions. 
\begin{assumption}\label{assumption1}
Assume that 
\begin{equation}\label{eq-m}
    X_{i}=\Gamma_1 U_{i}+\mu_{1}\ (i=1,\ldots,n),\quad
    Y_{j}=\Gamma_2 V_{j}+\mu_{2}\ (j=1,\ldots,m),
\end{equation}
where $\Gamma_i\in\mathbb{R}^{p\times q}$ with $\Gamma_1\Gamma_1^\top=\Sigma_1$, $\Gamma_2\Gamma_2^\top=\Sigma_2$, and $q$ could be bigger, equal to or smaller than $p$. Also, $\{U_{i}\}_{i=1}^{n}$ and $\{V_{j}\}_{j=1}^{m}$ are both $q$-variate i.i.d. random vectors satisfying that $E(U_i)=E(V_j)=0$ and $\text{var}(U_i)=\text{var}(V_j)=I_q.$ Furthermore, if we write $U_i=(U_i(1),\ldots,U_i(q))$ and $V_j=(V_j(1),\ldots,V_j(q))$, we assume that the entries $\{U_i(k)\}_k$ and $\{V_j(k)\}_k$ are independent (not necessarily identically distributed) and $\max_{1\leq k\leq q}\{E(U_{1}(k)^{8}),E(V_{1}(k)^{8})\}<\infty$.
\end{assumption}

\begin{assumption}\label{assumption2}
Under the model in Assumption \ref{assumption1}, suppose the operator norms of $\Sigma_{1}$ and $\Sigma_{2}$ remain bounded in $p$, i.e., there exists $K>0$, such that $\max\{\|\Sigma_{1}\|_{\rm op},\|\Sigma_{2}\|_{\rm op}\}\leq K$, for all $p$. Besides, we assume that $\text{tr}(\Sigma_{i})=\Theta(p)$, for $i=1,2$.
\end{assumption}

\begin{assumption}\label{assumption3}
$\|\mu_{i}\|_{2}^{2}=O(p)$ for $i=1,2$. Under Assumption \ref{assumption2}, we further have $\mu_{i}^{\top}\Sigma_{j}\mu_{i}\le\|\Sigma_{j}\|_{\rm op}\|\mu_{i}\|_{2}^{2}=O(p)$ for $i,j=1\text{ or }2$. 
\end{assumption}

\begin{assumption}\label{assumption4}
$\lim_{N\rightarrow\infty}n/N= \kappa\in(0,1)$, where $N=n+m$.
\end{assumption}

\begin{assumption}\label{assumption5}
Let $g(x)=f(x^{2})$, then $k(x,y)=g(\|x-y\|_2/\surd{p})$. We assume that $g$ is a $C^{3}$ function on $[0,+\infty)$, and $\sup_{1\leq s\leq 3}\sup_{x\geq 0}|g^{(s)}(x)|<\infty$.
\end{assumption}

Assumption \ref{assumption1} is common in the random matrix theory literature, and it has been used in the analysis of the spectrum of kernel random matrix \citep{el2010spectrum}, kernel ridgeless regression \citep{liang2020just}, and high-dimensional two-sample tests \citep{bai1996effect,chen2010two}. Assumption \ref{assumption2} requires the eigenvalues of $\Sigma_{1}$ and $\Sigma_{2}$ to be uniformly bounded from above, and the number of nonzero eigenvalues to be of the order $\Theta(p)$. However, we do not require the eigenvalues to be bounded away from zero. The assumption is trivially satisfied for data with independent entries and bounded variance. It also holds for data with, for example, \textsc{ar}$(1)$ structure or banded covariance, but not the compound symmetry structure as its largest eigenvalue is of order $\Theta(p)$. Assumption \ref{assumption3} ensures that the Taylor expansions below are around points that are asymptotically bounded. Assumption \ref{assumption5} covers the aforementioned Gaussian kernel, Laplace kernel, rational quadratic kernel, and the energy distance. 

Define
\[
\tau_{i}=2\frac{\text{tr}(\Sigma_{i})}{p}~\text{for }i=1,2,\quad \tau_{3}=\frac{\text{tr}(\Sigma_{1})+\text{tr}(\Sigma_{2})+\|\mu_{1}-\mu_{2}\|^{2}_{2}}{p}.
\]

As shown in Section \ref{secB.1} in the Appendix, we have the following second-order Taylor expansions under Assumption \ref{assumption5}:
\begin{align*}
    f(p^{-1}\|X_{i_1}-X_{i_2}\|_{2}^2)&=f(\tau_1)+f^{(1)}(\tau_1)\widetilde{X}_{i_1,i_2}+c_{2,\tau_1}(\widetilde{X}_{i_1,i_2})\widetilde{X}_{i_1,i_2}^{2},\\
    f(p^{-1}\|Y_{j_1}-Y_{j_2}\|_{2}^2)&=f(\tau_2)+f^{(1)}(\tau_2)\widetilde{Y}_{j_1,j_2}+c_{2,\tau_2}(\widetilde{Y}_{j_1,j_2})\widetilde{Y}_{j_1,j_2}^{2},\\
    f(p^{-1}\|X_{i}-Y_{j}\|_{2}^2)&=f(\tau_3)+f^{(1)}(\tau_3)\widetilde{Z}_{i,j}+c_{2,\tau_3}(\widetilde{Z}_{i,j})\widetilde{Z}_{i,j}^{2},
\end{align*}
with $\widetilde{X}_{i_1,i_2}=p^{-1}\|X_{i_{1}}-X_{i_{2}}\|_{2}^{2}-\tau_1$, $\widetilde{Y}_{j_1,j_2}=p^{-1}\|Y_{j_{1}}-Y_{j_{2}}\|_{2}^{2}-\tau_2$ and $\widetilde{Z}_{i,j}=p^{-1}\|X_{i}-Y_{j}\|_{2}^{2}-\tau_3$, where $c_{2,\tau_1}(\cdot)$, $c_{2,\tau_2}(\cdot)$ and $c_{2,\tau_3}(\cdot)$ are some bounded functions regardless of the underlying distributions $P_X$ and $P_Y$. These expansions are built on Lemma \ref{lemma-expansion} in the Appendix which generalizes Lemma 7 in \citet{gao2021asymptotic}. 

Plugging the above expansions into the unbiased MMD estimator (\ref{eq-mmds}), we obtain the following
\begin{equation}\label{MDD-expan}
    \text{MMD}_{n,m}^{2}=\Delta_{0}+\Delta_{1}+\widetilde{\Delta}_{2},
\end{equation}
where
\begin{align}
\Delta_{0}&=f(\tau_{1})+f(\tau_{2})-2f(\tau_{3}),\label{delta0-2} \\
\Delta_{1}&=\Delta_{1,1}+\Delta_{1,2}+\Delta_{1,3}, \nonumber \\
\widetilde{\Delta}_{2}&=\frac{1}{n(n-1)}\sum_{i_{1}\neq i_{2}}c_{2,\tau_1}(\widetilde{X}_{i_1,i_2})\widetilde{X}_{i_1,i_2}^{2}+\frac{1}{m(m-1)}\sum_{j_{1}\neq j_{2}}c_{2,\tau_2}(\widetilde{Y}_{j_1,j_2})\widetilde{Y}_{j_1,j_2}^{2}-\frac{2}{nm}\sum_{i,j}c_{2,\tau_3}(\widetilde{Z}_{i,j})\widetilde{Z}_{i,j}^{2}, \nonumber
\end{align}
As can be seen from the decomposition of $\Delta_{1}$, $\Delta_{1,2}$ only captures the mean differences, and $\Delta_{1,3}$ quantifies the discrepancy between both the means and traces of the covariance matrices after the nonlinear transformation $f^{(1)}$. Similar to $\Delta_{1,3}$, $\Delta_0$ quantifies the differences between the means and traces of the covariance matrices after the nonlinear transformation $f$. These three quantities vanish under the null hypothesis, while $\Delta_{1,1}$ does not. The remainder term $\widetilde{\Delta}_{2}$ encodes information of higher-order moment discrepancies. 
\begin{align*}
\Delta_{1,1}&=-\frac{2}{p}\bigg\{\frac{f^{(1)}(\tau_{1})}{n(n-1)}\sum_{i_{1}\neq i_{2}}(X_{i_{1}}-\mu_{1})^\top(X_{i_{2}}-\mu_{1})+\frac{f^{(1)}(\tau_{2})}{m(m-1)}\sum_{j_{1}\neq j_{2}}(Y_{j_{1}}-\mu_{2})^\top(Y_{j_{2}}-\mu_{2})\\
&\quad-\frac{2}{nm}f^{(1)}(\tau_{3})\sum_{i,j}(X_{i}-\mu_{1})^\top (Y_{j}-\mu_{2})\bigg\},\\
\Delta_{1,2}&=-\frac{4}{p}f^{(1)}(\tau_{3})\bigg\{\frac{1}{n}\sum_{i}(X_{i}-\mu_{1})^\top(\mu_{1}-\mu_{2})+\frac{1}{m}\sum_{j}(Y_{j}-\mu_{2})^\top (\mu_{2}-\mu_{1})\bigg\},\\
\Delta_{1,3}&=\frac{2}{p}\bigg[\{f^{(1)}(\tau_{1})-f^{(1)}(\tau_{3})\}\frac{1}{n}\sum_{i}\|X_{i}-\mu_{1}\|_{2}^{2}+\{f^{(1)}(\tau_{2})-f^{(1)}(\tau_{3})\}\frac{1}{m}\sum_{j}\|Y_{j}-\mu_{2}\|_{2}^{2}\bigg]\\
&\quad-2p^{-1}[\{f^{(1)}(\tau_{1})-f^{(1)}(\tau_{3})\}\text{tr}(\Sigma_{1})+\{f^{(1)}(\tau_{2})-f^{(1)}(\tau_{3})\}\text{tr}(\Sigma_{2})].
\end{align*}

Our analysis in Sections \ref{sec3.2}--\ref{sec3.3} below involves two major steps: 
1) showing the asymptotic normality of $\Delta_1/\surd{\text{var}(\Delta_1)}$ under the null, local and fixed alternatives;
2) proving that $(\widetilde{\Delta}_{2}-T_1)/\surd{\text{var}(\Delta_1)}=o_p(1)$ for some nonrandom quantity $T_1$ defined in (\ref{T1-2}) below. For the result in the second step to hold, we require $N$ to grow slower than certain polynomial order of $p$. When $N$ grows at a faster rate, we need to look into the higher-order terms of $\text{MMD}_{n,m}^2$. To this end, we shall consider the $l$th order Taylor expansions of $\text{MMD}_{n,m}^2$ for any $l\geq 3$ in Section \ref{sec3.4}.

\subsection{Non-null CLTs based on second-order Taylor series}\label{sec3.2}
The goal here is to derive the CLTs under the local and fixed alternatives based on the second-order Taylor expansion in (\ref{MDD-expan}). To begin with, we deal with the first-order term $\Delta_{1}$. Assume one of the following two conditions holds
\begin{assumption}\label{assumption6}
$(\mu_{1}-\mu_{2})^\top\Sigma_{i}(\mu_{1}-\mu_{2})=o(N^{-1}p)$ and $|f^{(1)}(\tau_{i})-f^{(1)}(\tau_{3})|=o(N^{-1/2})$ for $i=$1 and 2, 
\end{assumption}
\begin{assumption}\label{assumption7}
$(\mu_{1}-\mu_{2})^\top\Sigma_{i}(\mu_{1}-\mu_{2})=\omega(N^{-1}p)$ or $|f^{(1)}(\tau_{i})-f^{(1)}(\tau_{3})|=\omega(N^{-1/2})$ for $i=$1 or 2,
\end{assumption}
Then the asymptotic normality of the first-order term $\Delta_{1}$ can be established. 
\begin{lemma}
\label{lemma3.1}
Under Assumptions \ref{assumption1}--\ref{assumption4}, and either \ref{assumption6} or \ref{assumption7},  
\[
\frac{\Delta_{1}}{\surd{\textup{var}(\Delta_{1})}}\rightarrow N(0,1)\text{ in distribution, as }N\rightarrow\infty\text{ and }p\rightarrow\infty.
\]
\end{lemma}

As shown in the proof, under Assumption \ref{assumption6}, we have $\text{var}(\Delta_{1,2})=o(\text{var}(\Delta_{1,1}))$ and $\text{var}(\Delta_{1,3})=o(\text{var}(\Delta_{1,1}))$, which implies that $(\Delta_{1,2}+\Delta_{1,3})/\surd{\text{var}(\Delta_{1})}=o_{p}(1)$. In this case, the discrepancies between $\mu_1$ and $\mu_2$ and $\text{tr}(\Sigma_1)$ and $\text{tr}(\Sigma_2)$ are small enough so that $\Delta_{1,1}$ dominates the other two terms. Thus, we view Assumption \ref{assumption6} as a local alternative. On the contrary, Assumption \ref{assumption7} requires $\mu_1$ and $\mu_2$ and $\text{tr}(\Sigma_1)$ and $\text{tr}(\Sigma_2)$ to be far apart so as either $\Delta_{1,2}$ or $\Delta_{1,3}$ dominates $\Delta_{1,1}$, which can be regarded as the fixed alternative. 

Next we analyze the second-order term $\widetilde{\Delta}_{2}$. We work with the variance and the mean of $\widetilde{\Delta}_{2}$ separately. If we define
\begin{equation}\label{T1-2}
\begin{split}
    T_{1}=(2p^{2})^{-1}[&f^{(2)}(\tau_{1})E\{(\|X_{1}-X_{2}\|_{2}^{2}-p\tau_{1})^{2}\}+f^{(2)}(\tau_{2})E\{(\|Y_{1}-Y_{2}\|_{2}^{2}-p\tau_{2})^{2}\}
    \\&-2f^{(2)}(\tau_{3})E\{(\|X_{1}-Y_{1}\|_{2}^{2}-p\tau_{3})^{2}\}]
\end{split}
\end{equation}
then under Assumptions \ref{assumption1}--\ref{assumption5}, we can show that
\begin{align*}
\text{var}(\widetilde{\Delta}_{2})=O(N^{-1}p^{-2}),\quad E(\widetilde{\Delta}_{2})-T_{1}=O(p^{-3/2}).
\end{align*}
We refer the readers to Section \ref{secB.2} in the Appendix for the details of the calculations of the variance and the mean of $\widetilde{\Delta}_{2}$. Utilizing the above results, we have the following theorem which establishes the asymptotic normality of $\text{MMD}_{n,m}^{2}$ under the local alternative. 

\begin{theorem}
\label{theorem3.2}
Under Assumptions \ref{assumption1}--\ref{assumption6}, and $N=o(p)$,  
\[
\frac{\textup{MMD}_{n,m}^{2}-\Delta_{0}-T_{1}}{\surd{\textup{var}({\color{red}\Delta_{1,1}})}}\rightarrow N(0,1)\text{ in distribution, as }N\rightarrow\infty\text{ and }p\rightarrow\infty.
\]
\end{theorem}

One can show that $T_1=0$ if $\mu_1=\mu_2$ and $\Sigma_1=\Sigma_2$ (see Section \ref{secB.5} in the Appendix for the proof of a more general result). Hence, $T_1$ characterizes the discrepancy between the means and the covariance matrices, not only the traces. Since $T_{1}=O(p^{-1})$ under Assumptions \ref{assumption1}--\ref{assumption3}, we immediately have the following corollary.
\begin{corollary}
\label{corollary3.3}
Under Assumptions \ref{assumption1}--\ref{assumption6}, and $N=o(\surd{p})$,  
\[
\frac{\textup{MMD}_{n,m}^{2}-\Delta_{0}}{\surd{\textup{var}({\color{red}\Delta_{1,1}})}}\rightarrow N(0,1)\text{ in distribution, as }N\rightarrow\infty\text{ and }p\rightarrow\infty.
\]
\end{corollary}

Similarly, under the fixed alternative, we have the following theorem and corollary. 
\begin{theorem}
\label{theorem3.4}
Under Assumptions \ref{assumption1}--\ref{assumption5}, and for $i=1$ or $2$, either $(\mu_{1}-\mu_{2})^{\top}\Sigma_{i}(\mu_{1}-\mu_{2})=\omega(\max\{Np^{-1},N^{-1}p\})$ or $|f^{(1)}(\tau_{i})-f^{(1)}(\tau_{3})|=\omega(\max\{N^{1/2}p^{-1},N^{-1/2}\})$,
\[
\frac{\textup{MMD}_{n,m}^{2}-\Delta_{0}-T_{1}}{\surd{\textup{var}({\color{red}\Delta_{1,2}+\Delta_{1,3}})}}\rightarrow N(0,1)\text{ in distribution, as }N\rightarrow\infty\text{ and }p\rightarrow\infty.
\]
\end{theorem}
\begin{corollary}
\label{corollary3.5}
Under Assumptions \ref{assumption1}--\ref{assumption5}, and for $i=1$ or $2$, either $(\mu_{1}-\mu_{2})^{\top}\Sigma_{i}(\mu_{1}-\mu_{2})=\omega(\max\{N,N^{-1}p\})$ or $|f^{(1)}(\tau_{i})-f^{(1)}(\tau_{3})|=\omega(\max\{N^{1/2}p^{-1/2},N^{-1/2}\})$, 
\[
\frac{\textup{MMD}_{n,m}^{2}-\Delta_{0}}{\surd{\textup{var}({\color{red}\Delta_{1,2}+\Delta_{1,3}})}}\rightarrow N(0,1)\text{ in distribution, as }N\rightarrow\infty\text{ and }p\rightarrow\infty.
\]
\end{corollary}

Assumption \ref{assumption7} is already included in Theorem \ref{theorem3.4} and Corollary \ref{corollary3.5}. If we further assume either $(\mu_{1}-\mu_{2})^\top\Sigma_{i}(\mu_{1}-\mu_{2})=\Theta(N^{-1}p^{a})$ for some $a>1$ or $|f^{(1)}(\tau_{i})-f^{(1)}(\tau_{3})|=\Theta(N^{-1/2}p^{b})$ for some $b>0$, then the condition in Theorem \ref{theorem3.4} becomes either $N=o(p^{(a+1)/2})$ or $N=o(p^{b+1})$, and the condition in Corollary \ref{corollary3.5} becomes either $N=o(p^{a/2})$ or $N=o(p^{b+1/2})$.

\subsection{Null CLT and the test statistic}\label{sec3.3}
Although the asymptotic normality of $\text{MMD}_{n,m}^{2}$ under the null $H_0:P_X=P_Y$ can be seen as a special case of Theorem \ref{theorem3.2}, the analysis under the null can indeed be refined to obtain a sharper rate of the remainder term $\widetilde{\Delta}_{2}$. As a consequence, the restriction between the sample size $N$ and the data dimension $p$ can be dropped. To be specific, under $H_{0}$, we have
\[
\Delta_{0}=0,\quad
\Delta_{1,2}=\Delta_{1,3}=0\implies\Delta_{1}=\Delta_{1,1}.
\]
Assumption \ref{assumption6} is trivially satisfied, and Lemma \ref{lemma3.1} holds under Assumptions \ref{assumption1}--\ref{assumption4}. As shown in Section \ref{sec3.2}, the restriction between $N$ and $p$ in Theorem \ref{theorem3.2} arises from the variance and the mean of $\widetilde{\Delta}_{2}$. However, as $X$ and $Y$ have the same distribution under $H_{0}$, 
\[
E(\widetilde{\Delta}_{2})=E(\text{MMD}_{n,m}^{2})-\Delta_{0}-E(\Delta_{1})=0,
\]
where we have used the fact that $E(\text{MMD}_{n,m}^{2})=\text{MMD}^2(P_X,P_Y)=0.$ Moreover, under $H_{0}$, it can be proved that
\[
\text{var}(\widetilde{\Delta}_{2})=O(N^{-2}p^{-2}),
\]
the order of which is smaller than that in the non-null setting. See Section \ref{secB.3} in the Appendix. Combining these observations, we obtain the following theorem under $H_0$.

\begin{theorem}
\label{theorem3.6}
Under $H_0:P_X=P_Y$, and Assumptions \ref{assumption1}--\ref{assumption5}, 
\[
\frac{\textup{MMD}_{n,m}^{2}}{\surd{\textup{var}({\color{red}\Delta_{1,1}})}}\rightarrow N(0,1)\text{ in distribution, as }N\rightarrow\infty\text{ and }p\rightarrow\infty.
\]
\end{theorem}

In order to formulate a testing procedure based on Theorem \ref{theorem3.6}, $\text{var}(\Delta_{1,1})$ needs to be estimated. As shown in the proof of Lemma \ref{lemma3.1}, 
\[
\text{var}(\Delta_{1,1})=\frac{8}{p^{2}}\bigg[\frac{1}{n(n-1)}\{f^{(1)}(\tau_{1})\}^{2}\text{tr}(\Sigma_{1}^{2})+\frac{1}{m(m-1)}\{f^{(1)}(\tau_{2})\}^{2}\text{tr}(\Sigma_{2}^{2})+\frac{2}{nm}\{f^{(1)}(\tau_{3})\}^{2}\text{tr}(\Sigma_{1}\Sigma_{2})\bigg],
\]
which motivates the following plug-in estimator
\[
\widehat{\text{var}({\color{red}\Delta_{1,1}})}=\frac{8}{p^{2}}\bigg[\frac{1}{n(n-1)}\{f^{(1)}(\widehat{\tau_{1}})\}^{2}\widehat{\text{tr}(\Sigma_{1}^{2})}+\frac{1}{m(m-1)}\{f^{(1)}(\widehat{\tau_{2}})\}^{2}\widehat{\text{tr}(\Sigma_{2}^{2})}+\frac{2}{nm}\{f^{(1)}(\widehat{\tau_{3}})\}^{2}\widehat{\text{tr}(\Sigma_{1}\Sigma_{2})}\bigg],
\]
where
\begin{align*}
    \widehat{\tau_{1}}&=\frac{2}{p}\times\frac{1}{n}\sum_{i}X_{i}^{\top}(X_{i}-\bar{X}_{-i}),\quad \widehat{\tau_{2}}=\frac{2}{p}\times\frac{1}{m}\sum_{j}Y_{j}^{\top}(Y_{j}-\bar{Y}_{-j}),\\
    \widehat{\tau_{3}}=\frac{1}{p}\times\bigg\{&\frac{1}{n}\sum_{i}X_{i}^{\top}(X_{i}-\bar{X}_{-i})+\frac{1}{m}\sum_{j}Y_{j}^{\top}(Y_{j}-\bar{Y}_{-j})\\
    &+\frac{1}{n(n-1)}\sum_{i_{1}\neq i_{2}}X_{i_{1}}^{\top}X_{i_{2}}+\frac{1}{m(m-1)}\sum_{j_{1}\neq j_{2}}Y_{j_{1}}'Y_{j_{2}}-\frac{2}{nm}\sum_{i,j}X_{i}^{\top}Y_{j}\bigg\},
\end{align*}
with $\bar{X}_{-i}$ ($\bar{Y}_{-j}$) being the sample mean excluding $X_{i}$ ($Y_j$). The definitions of $\widehat{\text{tr}(\Sigma_{1}^{2})}$, $\widehat{\text{tr}(\Sigma_{2}^{2})}$ and $\widehat{\text{tr}(\Sigma_{1}\Sigma_{2})}$ can be found on page 7 of \citet{chen2010two}. The above variance estimator is designed to estimate the leading term of the variance of the sample MMD under the {\color{red}null}. It appears to be new and is different from the ones recently proposed in the literature, see, e.g., \citet{chakraborty2021new} and \citet{gao2021two}.

The next theorem establishes the ratio consistency of {\color{red}$\widehat{\textup{var}(\Delta_{1,1})}$, which holds under both the null and the alternative}. 
\begin{theorem}
\label{theorem3.7}
Under Assumptions \ref{assumption1}--{\color{red}\ref{assumption5}}, 
\[
\frac{\widehat{\textup{var}({\color{red}\Delta_{1,1}})}}{\textup{var}({\color{red}\Delta_{1,1}})}\rightarrow 1\text{ in probability, as }N\rightarrow\infty\text{ and }p\rightarrow\infty.
\]
\end{theorem}

Based on Theorems \ref{theorem3.6}--\ref{theorem3.7}, we propose a test for $H_0$ as follows: at level $\alpha\in(0,1)$, reject $H_0$ if 
\[
\textup{MMD}_{n,m}^{2}/\surd{\widehat{\textup{var}({\color{red}\Delta_{1,1}})}}>z_{1-\alpha},
\]
and fail to reject $H_0$ otherwise, where $z_{1-\alpha}$ is the $100(1-\alpha)$th quantile of $N(0,1)$.

\subsection{Non-null CLTs based on higher-order Taylor series}\label{sec3.4}
Theorems \ref{theorem3.2}--\ref{theorem3.4} indicate that when $N$ grows slower than certain order of $p$, the kernel two-sample test is capable of detecting the discrepancy between the mean vectors and the covariance matrices through the quantities $\Delta_0$ and $T_1$. A natural question to ask is what high-order features are captured by the MMD if $N$ grows faster than the rates specified in Theorems \ref{theorem3.2}--\ref{theorem3.4}. To answer this question, we perform a more in-depth analysis based on the high-order Taylor series. We need to strengthen the moment assumption in Assumptions \ref{assumption1} and \ref{assumption5} as follows: for some $l\geq 3$,
\begin{assumptionp}{assumption1}\label{assumption1p}
The other parts in Assumption \ref{assumption1} remain the same. We additionally require that $\max_{1\leq k\leq q}\{E(U_{1}(k)^{4l}),E(V_{1}(k)^{4l})\}<\infty;$
\end{assumptionp}
\begin{assumptionp}{assumption5}\label{assumption5p}
Assume $g$ is a $C^{l}$ function on $[0,+\infty)$, and $\sup_{1\leq s\leq l}\sup_{x\geq 0}|g^{(s)}(x)|<\infty$.
\end{assumptionp}

Consider the $l$th order Taylor expansion (see Section \ref{secB.1} in the Appendix):
\[
f(p^{-1}\|X_{i_1}-X_{i_2}\|_{2}^2)=f(\tau_1)+\sum^{l-1}_{s=1}\frac{f^{(s)}(\tau_1)}{s!}\widetilde{X}_{i_1,i_2}^s+c_{l,\tau_1}(\widetilde{X}_{i_1,i_2})\widetilde{X}_{i_1,i_2}^{l},
\]
where $c_{l,\tau_1}(\cdot)$ is some bounded function. Analogous expansions hold for $f(p^{-1}\|Y_{j_1}-Y_{j_2}\|_{2}^{2})$ and $f(p^{-1}\|X_{i}-Y_{j}\|_{2}^{2})$. Then the unbiased MMD estimator (\ref{eq-mmds}) can be decomposed as
\[
\text{MMD}_{n,m}^{2}=\Delta_{0}+\sum^{l-1}_{s=1}\Delta_{s}+\widetilde{\Delta}_{l},
\]
where
\begin{align*}
    \Delta_{s}&=\frac{1}{s!}\bigg\{\frac{f^{(s)}(\tau_1)}{n(n-1)}\sum_{i_{1}\neq i_{2}}\widetilde{X}_{i_1,i_2}^s+\frac{f^{(s)}(\tau_2)}{m(m-1)}\sum_{j_{1}\neq j_{2}}\widetilde{Y}_{j_1,j_2}^s-\frac{2}{nm}f^{(s)}(\tau_3)\sum_{i,j}\widetilde{Z}_{i,j}^s\bigg\}\ \text{ for }s=1,\ldots,l-1,\\
    \widetilde{\Delta}_{l}&=\frac{1}{n(n-1)}\sum_{i_{1}\neq i_{2}}c_{l,\tau_1}(\widetilde{X}_{i_1,i_2})\widetilde{X}_{i_1,i_2}^{l}+\frac{1}{m(m-1)}\sum_{j_{1}\neq j_{2}}c_{l,\tau_2}(\widetilde{Y}_{j_1,j_2})\widetilde{Y}_{j_1,j_2}^{l}-\frac{2}{nm}\sum_{i,j}c_{l,\tau_3}(\widetilde{Z}_{i,j})\widetilde{Z}_{i,j}^{l}.
\end{align*}

Following similar arguments as in the calculation of $\text{var}(\widetilde{\Delta}_{2})$, one can show
\[
\text{var}(\widetilde{\Delta}_{l})=O(N^{-1}p^{-l}).
\]
In addition, we assume that
\begin{assumption}\label{assumption8}
for $s=1,\ldots,l-1$, $\text{var}\{f^{(s)}(\tau_1)E(\widetilde{X}_{1,2}^s\mid X_1)-f^{(s)}(\tau_3)E(\widetilde{Z}_{1,1}^s\mid X_1)\}=o(N^{-1}p^{-1})$ and $\text{var}\{f^{(s)}(\tau_2)E(\widetilde{Y}_{1,2}^s\mid Y_1)-f^{(s)}(\tau_3)E(\widetilde{Z}_{1,1}^s\mid Y_1)\}=o(N^{-1}p^{-1}),$ 
\end{assumption}
Then we have
\[
\text{var}(\Delta_s)=o(N^{-2}p^{-1})\text{ for }s=2,\ldots,l-1.
\]
See more details for the analysis of the variances of $\Delta_{s}$ with $s=2,\ldots,l-1$ and $\widetilde{\Delta}_{l}$ in Section \ref{secB.4} in the Appendix. 

Assumption \ref{assumption8} imposes a type of local alternative ensuring that {\color{red}$\Delta_1$
(or more precisely $\Delta_{1,1}$) is the dominant term} of the sample MMD, which can be viewed as a generalization of Assumption \ref{assumption6}. To illustrate Assumption \ref{assumption8}, we present the following two results.
\begin{proposition}\label{prop1}
For $s=1$, Assumption \ref{assumption8} is equivalent to Assumption \ref{assumption6}. 
\end{proposition}
\begin{proposition}\label{prop2}
For $s=2$, suppose $\mu_1=\mu_2=\mu$ and $\Sigma_1=\Sigma_2$. Under Assumptions \ref{assumption1}--\ref{assumption2}, 
\[
|\textup{var}\{f^{(2)}(\tau_1)E(\widetilde{X}_{1,2}^2\mid X_1)-f^{(2)}(\tau_3)E(\widetilde{Z}_{1,1}^2\mid X_1)\}|=O\bigg(p^{-4}\sum_{k=1}^{q}\{s_{kk}(\mu_{k,3}^{(1)}-\mu_{k,3}^{(2)})\}^2\bigg),
\]
where $\mu_{k,3}^{(1)}=E(U_1(k)^3)$, $\mu_{k,3}^{(2)}=E(V_1(k)^3)$ and $\Gamma^\top \Gamma=(s_{kl})_{q\times q}$. 
\end{proposition}

As a result of Proposition \ref{prop2}, Assumption \ref{assumption8} holds for $s=2$ provided that $\sum_{k=1}^{q}\{s_{kk}(\mu_{k,3}^{(1)}-\mu_{k,3}^{(2)})\}^2=o(N^{-1}p^{3})$, which holds if $\max_k|\mu_{k,3}^{(1)}-\mu_{k,3}^{(2)}|\leq C$ and $N=o(p^2)$ for some constant $C>0.$

\begin{theorem}\label{theorem3.9}
Under Assumptions \ref{assumption1p}, \ref{assumption2}--\ref{assumption4}, \ref{assumption5p}, \ref{assumption8} and $N=o(p^{l-1})$,
\[
\frac{\textup{MMD}_{n,m}^{2}-\textup{MMD}^2(P_X,P_Y)}{\surd{\textup{var}({\color{red}\Delta_{1,1}})}}\rightarrow N(0,1)\text{ in distribution, as }N\rightarrow\infty\text{ and }p\rightarrow\infty.
\]
\end{theorem}

Below we study the effects from the higher-order moments. To disentangle the higher-order effects, we assume throughout the following discussions that
$X$ and $Y$ share the first $(l-1)$ moments, i.e., $\mu_{1}=\mu_{2}$, $\Sigma_{1}=\Sigma_{2}$, $E(U_{1}(k)^{s})=E(V_{1}(k)^{s})$ for $s=1,\ldots,l-1$ and $k=1,\ldots,q$. We begin by analyzing the higher-order effects on the population MMD. Define
\begin{equation}\label{eq-Ts}
T_{s-1}=\frac{f^{(s)}(\tau)}{s!p^{s}}\left[E\{(\|X_{1}-X_{2}\|_{2}^{2}-p\tau)^{s}\}+E\{(\|Y_{1}-Y_{2}\|_{2}^{2}-p\tau)^{s}\}-2E\{(\|X_{1}-Y_{1}\|_{2}^{2}-p\tau)^{s}\}\right]
\end{equation}
for $s\geq 2$, where $\tau=\tau_1=\tau_2=\tau_3$. 
\begin{remark}\label{remark1}
Suppose all the moments of $U_i$ and $V_j$ exist. Applying the Taylor series expansion to the population MMD, we get 
$\text{MMD}^2(P_X,P_Y)=\Delta_0+\sum^{\infty}_{s=1}T_{s},$
as $p\rightarrow \infty.$ It suggests that the combined effect from $\Delta_0$ and all $T_s$'s is equivalent to that of the population MMD. When the first $(l-1)$ moments of $X$ and $Y$ coincide, $\Delta_0=0$ and $T_s=0$ for $s=1,\ldots,l-2$, which is shown in Section \ref{secB.5} in the Appendix. Therefore, 
$\text{MMD}^2(P_X,P_Y)=\sum^{\infty}_{s=l-1}T_{s},$
which indicates that MMD detects the discrepancy lying on the higher-order moments through the quantity $\sum^{\infty}_{s=l-1}T_{s}$.
\end{remark}

The expansion in Remark \ref{remark1} motivates us to study the properties of $T_{s}$ for $s\geq l-1$. For any $a_1\ge 0,a_2\ge 0,a_1+a_2\leq a$, define two tensors
\[
\mathcal{T}_{1,a}^{(a_1,a_2)}=(\mu_{i_1,\ldots,i_{a-a_1-a_2}}^{(1)})_{1\leq i_1,\ldots,i_{a-a_1-a_2}\leq p},\quad\mathcal{T}_{2,a}^{(a_1,a_2)}=(\mu_{i_1,\ldots,i_{a-a_1-a_2}}^{(2)})_{1\leq i_1,\ldots,i_{a-a_1-a_2}\leq p},
\]
where 
\[
\mu_{i_1,\ldots,i_{a-a_1-a_2}}^{(1)}=\sum_{j_1,\ldots,j_{a_1}}E\bigg[\prod^{a-a_1-a_2}_{s=1}\{x_{i_s}-E(x_{i_s})\}\prod^{a_1}_{s=1}\{x_{j_s}-E(x_{j_s})\}^2\bigg],
\]
and $\mu_{i_1,\ldots,i_{a-a_1-a_2}}^{(2)}$ is defined in a similar fashion with respect to $Y$. 

\begin{lemma}\label{lemma3.8}
Under $\mu_{1}=\mu_{2}$, $\Sigma_{1}=\Sigma_{2}$, $E(U_{1}(k)^{s})=E(V_{1}(k)^{s})$ for $s=1,\ldots,l-1$ and $k=1,\ldots,q$, we have for any $l\geq 3$,
\[
 T_{l-1}=\frac{f^{(l)}(\tau)}{p^{l}}\sum_{\substack{0\leq a_1+a_2\leq l\\a_1=a_2}}(-2)^{l-a_1-a_2}\frac{1}{a_1!a_2!(l-a_1-a_2)!}\|\mathcal{T}_{1,l}^{(a_1,a_2)}-\mathcal{T}_{2,l}^{(a_1,a_2)}\|_{\rm F}^2.  
\]
Besides, $T_{l-1}=O(p^{-2r+2})$ when $l=2r-1$ or $l=2r$ for $r\ge 2$. Assume additionally that
\begin{assumption}\label{assumption9}
$\|\mathcal{T}_{1,a}^{(a_1,a_2)}-\mathcal{T}_{2,a}^{(a_1,a_2)}\|_{\rm F}^2=O(p^{a+a_1-a_2-2r+2})$ when $l=2r-1$ or $l=2r$ for $r \geq 2$, $a_1\ge 0$, $a_2\ge 0$, $a_1+a_2\leq a$, $|a_1-a_2|\leq a-l$, and $l+1\leq a\leq s\leq 2(l-1)-1$, which holds automatically when $X_i=U_i+\mu$ and $Y_j=V_j+\mu$. 
\end{assumption}
Then for any $l+1\leq s\leq 2(l-1)-1$ and $r\geq 2$, $T_{s-1}=O(p^{-2r+2})$ when $l=2r-1$ or $l=2r$. 
Consequently, for $l=2r-1$ or $l=2r$ with $r\geq 2$,
\[
\textup{MMD}^2(P_X,P_Y)=O(p^{-2r+2}).
\]
\end{lemma}

The term $\mathcal{T}_{1,a}^{(a_1,a_2)}-\mathcal{T}_{2,a}^{(a_1,a_2)}$ quantifies the discrepancy in the $(a+a_1-a_2)$th moment. The Frobenius norm in Assumption \ref{assumption9} is the cumulative effects from these discrepancies measured by $T_{s-1}$ for $l+1\leq s\leq 2(l-1)-1$. In short, Assumption \ref{assumption9} ensures that the discrepancy in the $s$th order moments with $s\geq l+1$ will have the same effect size or smaller effect size than that in the $l$th order moment. 

\begin{corollary}\label{Coro-trivial-1}
Suppose the assumptions in Theorem \ref{theorem3.9} and Lemma \ref{lemma3.8} hold. For $r\geq 2$, if $N=o(p^{l-3/2})$ when $l=2r-1$, or $N=o(p^{l-5/2})$ when $l=2r$, 
\[
\frac{\textup{MMD}_{n,m}^{2}}{\surd{\textup{var}({\color{red}\Delta_{1,1}})}}\rightarrow N(0,1)\text{ in distribution, as }N\rightarrow\infty\text{ and }p\rightarrow\infty.
\]
\end{corollary}

\begin{remark}
{\color{red}An intriguing observation from Lemma \ref{lemma3.8} is that the discrepancy in the $(l-1)$th order moment has the same effect size as that in the $l$th order moment, where $l=2r$ and $r\ge 2$. }From Corollary \ref{Coro-trivial-1}, when $N=o(p^{2r-{\color{red}1/2}})$ and the first ${\color{red}2r}$ moments of $X$ and $Y$ are equal for $r\geq {\color{red}1}$, the MMD-based test has trivial power. Put differently, the MMD-based test can only detect the discrepancy lying on the first ${\color{red}2r}$ moments when $N=o(p^{2r-{\color{red}1/2}})$. 
\end{remark}

\section{Power analysis}\label{sec:power-inner}
\subsection{When is the power trivial?}\label{sec:power-trivial}
Based on the CLT results, we study the power behavior of the test under the local alternative. We focus on two scenarios: 

S1. There is a discrepancy among the first two moments which satisfies Assumption \ref{assumption6}.

S2. The first $(l-1)$ moments are equal and Assumptions \ref{assumption8}--\ref{assumption9} hold, while there is a discrepancy between the $l$th moments for $l\geq 3$. 

Under S1, in view of Theorem \ref{theorem3.2}, the asymptotic power function of the test is
\begin{equation}\label{power1}
\Phi\bigg(-z_{1-\alpha}+\frac{\Delta_0+T_1}{\surd{\textup{var}({\color{red}\Delta_{1,1}})}}\bigg).    
\end{equation}
where $\Phi$ is the cdf of $N(0,1)$. We observe that:

O1-S1. When $N=o(\surd{p})$ and $\Delta_0=o(p^{-1/2}N^{-1}),$ or 
$N=o(p)$ and $\Delta_0+T_1=o(p^{-1/2}N^{-1}),$ the test has trivial power;

O2-S1. When $N=o(\surd{p})$ and $\Delta_0=\Theta(p^{-1/2}N^{-1}),$ or $N=o(p)$ and $\Delta_0+T_1=\Theta(p^{-1/2}N^{-1})$, the test has nontrivial power between $(0,1)$; 

O3-S1. When $N=o(\surd{p})$ and $\Delta_0=\omega(p^{-1/2}N^{-1}),$ or $N=o(p)$ and $\Delta_0+T_1=\omega(p^{-1/2}N^{-1})$, the test has power approaching one asymptotically.

Under S2, by Theorem \ref{theorem3.9}, the corresponding power function is equal to
\begin{align}\label{power2}
\Phi\bigg(-z_{1-\alpha}+\frac{\text{MMD}^2(P_X,P_Y)}{\surd{\textup{var}({\color{red}\Delta_{1,1}})}}\bigg),\quad l\geq 3.    
\end{align}
We have the following observations:

O1-S2. When $N=o(p^{l-1})$ and $\text{MMD}^2(P_X,P_Y)=o(p^{-1/2}N^{-1}),$ the test has trivial power;

O2-S2. When $N=o(p^{l-1})$ and $\text{MMD}^2(P_X,P_Y)=\Theta(p^{-1/2}N^{-1})$, the test has nontrivial power between $(0,1)$;
    
O3-S2. When $N=o(p^{l-1})$ and $\text{MMD}^2(P_X,P_Y)=\omega(p^{-1/2}N^{-1})$, the test has power approaching one asymptotically.

\begin{table}
\centering
\footnotesize
\def~{\hphantom{0}}
\caption{When do the MMD-based tests have trivial power? }{%
\begin{tabular}{ll}
%\\
\\
Dimension and sample size orders & Common features \\
fixed $N$, growing $p$ &  $\mu_1=\mu_2,\text{tr}(\Sigma_1)=\text{tr}(\Sigma_2)$ \\
$N=o(\surd{p})$ & $\mu_1=\mu_2,\text{tr}(\Sigma_1)=\text{tr}(\Sigma_2)$ \\
$N=o(p^{3/2})$ & $\mu_1=\mu_2,\Sigma_1=\Sigma_2$ and Assumption \ref{assumption8} holds\\
$N=o(p^{2r-{\color{red}1/2}})$ & the first ${\color{red}2r}$ moments are equal and Assumptions \ref{assumption8}--\ref{assumption9} hold\\
$N=o(p)$ & $\Delta_0+T_1=o(p^{-1/2}N^{-1})$ and Assumption \ref{assumption6} holds\\
$N=o(p^{l-1})$ & $\text{MMD}^2(P_X,P_Y)=o(p^{-1/2}N^{-1})$ and Assumption \ref{assumption8} holds
\end{tabular}}
\label{tb2}
\end{table}

As shown in Section \ref{secB.5} in the Appendix, the equality of the first $s$ moments is a sufficient condition for $T_{s-1}=0$. 
Regarding $\Delta_0$, we have the following necessary and sufficient condition for $\Delta_0=0.$
\begin{lemma}\label{lemma-delta0}
Suppose $f$ is strictly decreasing and strictly convex on $[0,+\infty)$. Then
\[
\Delta_{0}=0\text{ if and only if }\mu_{1}=\mu_{2}\text{ and }\textup{tr}(\Sigma_1)=\textup{tr}(\Sigma_2).
\]
\end{lemma}
As one can verify, the strict monotonicity and convexity assumption on $f$ is satisfied for the Gaussian/Laplace/rational quadratic kernels and the energy distance.

The result in O1 complements the finding in Lemma 4.1 and Theorem 4.2 of \citet{chakraborty2021new} which state that with $n,m$ fixed and $p\rightarrow+\infty$, the test based on energy distance has trivial power whenever $\mu_1=\mu_2 \text{ and } \text{tr}(\Sigma_1)=\text{tr}(\Sigma_2)$. As seen from O1-S1, the MMD-based test with the kernel satisfying the condition in Lemma \ref{lemma-delta0} remains to have trivial power under $\mu_1=\mu_2 \text{ and } \text{tr}(\Sigma_1)=\text{tr}(\Sigma_2)$ as long as $N=o(\surd{p})$. Moreover, some algebra shows that under $\mu_1=\mu_2 \text{ and } \text{tr}(\Sigma_1)=\text{tr}(\Sigma_2)$,
\begin{equation}\label{T1_2}
T_1=2p^{-2}f^{(2)}(\tau)\|\Sigma_1-\Sigma_2\|_{\rm F}^2.   
\end{equation}
When $N$ grows faster than $o(\surd{p})$ but slower than $p$, our result suggests that the power of the MMD-based test could be either trivial or non-trivial depending on the magnitude of the Frobenius norm of the difference between the two covariance matrices. 

Under $\mu_1=\mu_2$ and $\Sigma_1=\Sigma_2$, we note that 
$\Delta_0=0$, $E(\Delta_2)=T_1=0$, $E(\widetilde{\Delta}_3)-T_2=O(p^{-2})$, $\text{var}(\Delta_2)=o(N^{-2}p^{-1})$, and $\text{var}(\widetilde{\Delta}_3)=O(N^{-1}p^{-3})$. In the regime $N=o(p^{3/2})$, the asymptotic power function of the MMD-based test is thus
$\Phi(-z_{1-\alpha}+T_2/\surd{\textup{var}({\color{red}\Delta_{1,1}})}). $
By Lemma \ref{lemma3.8}, we have $T_2=O(p^{-2})$. Thus when Assumption \ref{assumption8} holds and $N=o(p^{3/2})$, $\mu_1=\mu_2$, $\Sigma_1=\Sigma_2,$ the MMD-based test has trivial power. In other words, when the discrepancy is lying on the third order or higher order moments, it is necessary for $N$ to grow at least as fast as $p^{3/2}$ for the test to have non-trivial power. 

More generally, by Corollary \ref{Coro-trivial-1}, when Assumptions \ref{assumption8}--\ref{assumption9} hold, $N=o(p^{2r-{\color{red}1/2}})$, and the first ${\color{red}2r}$ moments of $X$ and $Y$ are equal for $r\geq {\color{red}1}$, the MMD-based test has trivial power.

\subsection{When is the power approaching one?}
It may seem counter-intuitive to have an upper bound on the growth rate of $N$ (i.e., $N=o(p)$ and $N=o(p^{l-1})$) for the test to have nontrivial power in O2 and O3. These restrictions are from Theorems \ref{theorem3.2} and \ref{theorem3.9}, which enable us to get the asymptotic exact forms of the power functions under these two scenarios. The restrictions are certainly unnecessary for the test to have nontrivial power provided that the discrepancy between the two distributions is large enough. To see this, suppose there is a discrepancy between the first two moments that satisfies Assumption \ref{assumption6}. Consider the expansion
\[
\frac{\textup{MMD}_{n,m}^{2}}{{\color{red}\surd{\widehat{\textup{var}(\Delta_{1,1}})}}}=\frac{\Delta_0+T_1+\Delta_1+\{\widetilde{\Delta}_{2}-E(\widetilde{\Delta}_{2})\}+\{E(\widetilde{\Delta}_{2})-T_1\}}{\surd{\text{var}({\color{red}\Delta_{1,1}})}}{\color{red}\times\{1+o_{p}(1)\}}.   
\]
Recall that $\text{var}(\Delta_{1})=O(p^{-1}N^{-2})$, $\text{var}(\widetilde{\Delta}_{2})=O(p^{-2}N^{-1})$, and $E(\widetilde{\Delta}_{2})-T_1=O(p^{-3/2})$.
If 
\begin{align}\label{eq-pow-1}
\Delta_0+T_1=\omega(\max\{p^{-1/2}N^{-1},p^{-1}N^{-1/2},p^{-3/2}\}), 
\end{align}
then the asymptotic power converges to one. In particular, when $p$ and $N$ are of the same order and the difference lies on the means, the above rate is not only sufficient but also necessary in view of equation (3.11) of \citet{chen2010two} which suggests that under the local alternatives, the power will approach one only if $\|\mu_1-\mu_2\|^2=\omega(p^{1/2}N^{-1}).$ Note that for linear kernel (i.e., $f(x)=-x$), $\Delta_0=2\|\mu_1-\mu_2\|^2/p$, $T_1=0$ and the condition in (\ref{eq-pow-1}) is consistent with theirs. As the order of $\Delta_0$ can be as larger as $\Theta(1)$, we do not need additional restriction on the dimension-and-sample orders
for (\ref{eq-pow-1}) to be satisfied.

\begin{table}
\centering
\footnotesize
\def~{\hphantom{0}}
\caption{When do the MMD-based tests have non-trivial power? }{%
\begin{tabular}{ll}
%\\
\\
Dimension-and-sample orders & Discrepancy between $P_X$ and $P_Y$ \\
$N=o(\surd{p})$ & $\Delta_0=\Omega(p^{-1/2}N^{-1})$ and Assumption \ref{assumption6} holds\\
$N=o(p)$  & $\Delta_0+T_1=\Omega(p^{-1/2}N^{-1})$ and Assumption \ref{assumption6} holds\\
\text{no restriction}  &  $\Delta_0+T_1=\omega(\max\{p^{-1/2}N^{-1},p^{-1}N^{-1/2},p^{-3/2}\})$ and Assumption \ref{assumption6} holds\\
$N=o(p^{l-1})$ & $\text{MMD}^2(P_X,P_Y)=\Omega(p^{-1/2}N^{-1})$ and Assumption \ref{assumption8} holds\\
no restriction & $\text{MMD}^2(P_X,P_Y)=\omega(\max\{p^{-1/2}N^{-1},N^{-1/2}p^{-l/2}\})$ and Assumption \ref{assumption8} holds 
\end{tabular}}
\label{tb3}
\end{table}

More generally, as $\text{MMD}^2(P_X,P_Y)=E(\textup{MMD}_{n,m}^{2})=\Delta_0+\sum_{s=2}^{l-1}E(\Delta_s)+E(\widetilde{\Delta}_l)$, we have the expansion
\[
\frac{\textup{MMD}_{n,m}^{2}}{{\color{red}\surd{\widehat{\textup{var}(\Delta_{1,1}})}}}=\frac{\Delta_1+\sum^{l-1}_{s=2}\{\Delta_s-E(\Delta_s)\}+\{\widetilde{\Delta}_l-E(\widetilde{\Delta}_l)\}+\text{MMD}^2(P_X,P_Y)}{\surd{\textup{var}({\color{red}\Delta_{1,1}})}}{\color{red}\times\{1+o_{p}(1)\}}.
\]
Then under Assumption \ref{assumption8}, the power approaches one asymptotically whenever 
\[
\text{MMD}^2(P_X,P_Y)=\omega(\max\{p^{-1/2}N^{-1},p^{-l/2}N^{-1/2}\}).
\]

\section{Simulations}\label{sec:sim}
\subsection{Accuracy of the normal approximation under the null}
We perform simulation studies to validate the theoretical results established in previous sections. We consider the balanced case $n=m$. As for the kernel $k$, we consider the Gaussian kernel $k(x,y)=\exp(-\|x-y\|_{2}^{2}/\gamma)$, Laplace kernel $k(x,y)=\exp(-\|x-y\|_{2}/\surd{\gamma})$ and the energy distance (up to a scaling factor $\surd{\gamma}$) $k(x,y)=-\|x-y\|_{2}/\surd{\gamma}$ with the bandwidth $\gamma=2p$. We start by verifying the null CLT derived in Section \ref{sec3.3}.
\begin{example}\label{example1}
    We generate i.i.d. samples $X_i=\Sigma^{1/2}U_i+\mu$ for $i=1,\ldots,n$, and $Y_j=\Sigma^{1/2}V_j+\mu$ for $j=1,\ldots,m$ from the following models:
    
    1. $U_i(k)\stackrel{i.i.d.}{\sim}N(0,1)$ and $V_j(k)\stackrel{i.i.d.}{\sim}N(0,1)$ for $k=1,\ldots,p$. Also, $\mu=\mathbf{0}_{p\times 1}=(0,\ldots,0)^{\top}$ and $\Sigma=(\sigma_{ij})_{i,j=1}^{p}$ with $\sigma_{ij}=0.5^{|i-j|}$. 
    
    2. $U_i(k)\stackrel{i.i.d.}{\sim}\text{Poisson}(1)-1$ and $V_j(k)\stackrel{i.i.d.}{\sim}\text{Poisson}(1)-1$ for $k=1,\ldots,p$. Also, $\mu=\mathbf{1}_{p\times 1}=(1,\ldots,1)^{\top}$ and $\Sigma=(\sigma_{ij})_{i,j=1}^{p}$ with $\sigma_{ij}=0.5^{|i-j|}$. 
    
    3. $U_i(k)\stackrel{i.i.d.}{\sim}\text{Exponential}(1)-1$ and $V_j(k)\stackrel{i.i.d.}{\sim}\text{Exponential}(1)-1$ for $k=1,\ldots,p$. Also, $\mu=\mathbf{0}_{p\times 1}$ and $\Sigma=(\sigma_{ij})_{i,j=1}^{p}$ with $\sigma_{ii}=1$ for $i=1,\ldots,p$, $\sigma_{ij}=0.25$ if $1\le|i-j|\le 2$ and $\sigma_{ij}=0$ otherwise. 
\end{example}
Three different high-dimensional settings $(n,p)\in\{(32,1000),(200,200),(1000,100)\}$ are considered, where $n\approx p^{0.5},p,p^{1.5}$ respectively. We compare the standard normal quantiles with the corresponding sample quantiles of the test statistics. As can be seen from Figure \ref{fig:null} in the Appendix, the normal approximation appears to be quite accurate in all cases. 

\subsection{Power behavior}
Next, we analyze the power behavior of the test in the case of Theorem \ref{theorem3.2} where there is a discrepancy among the first two moments. For a given model, we can compute $\Delta_0$, $T_1$ and hence the power function (\ref{power1}). We consider the following two examples. In Example \ref{example2}, the mean vectors differ and the covariance matrices are equal. In Example \ref{example3}, the mean vectors are the same while the covariance matrices differ. 
\begin{example}\label{example2}
$X_i\stackrel{i.i.d.}{\sim}N(\mathbf{0}_{p\times 1},I_{p})$ for $i=1,\ldots,n$, and $Y_j\stackrel{i.i.d.}{\sim}N((2p)^{-1/2}\mathbf{1}_{p\times 1},I_{p})$ for $j=1,\ldots,m$. 
\end{example}
\begin{example}\label{example3}
$X_i\stackrel{i.i.d.}{\sim}N(\mathbf{0}_{p\times 1},I_{p})$ for $i=1,\ldots,n$, and $Y_j\stackrel{i.i.d.}{\sim}N(\mathbf{0}_{p\times 1},\Sigma)$ for $j=1,\ldots,m$, where $\Sigma=(\sigma_{ij})_{i,j=1}^{p}$ with $\sigma_{ij}=0.7^{|i-j|}$.
\end{example}
\begin{figure}
    \centering
    \includegraphics[width=\textwidth]{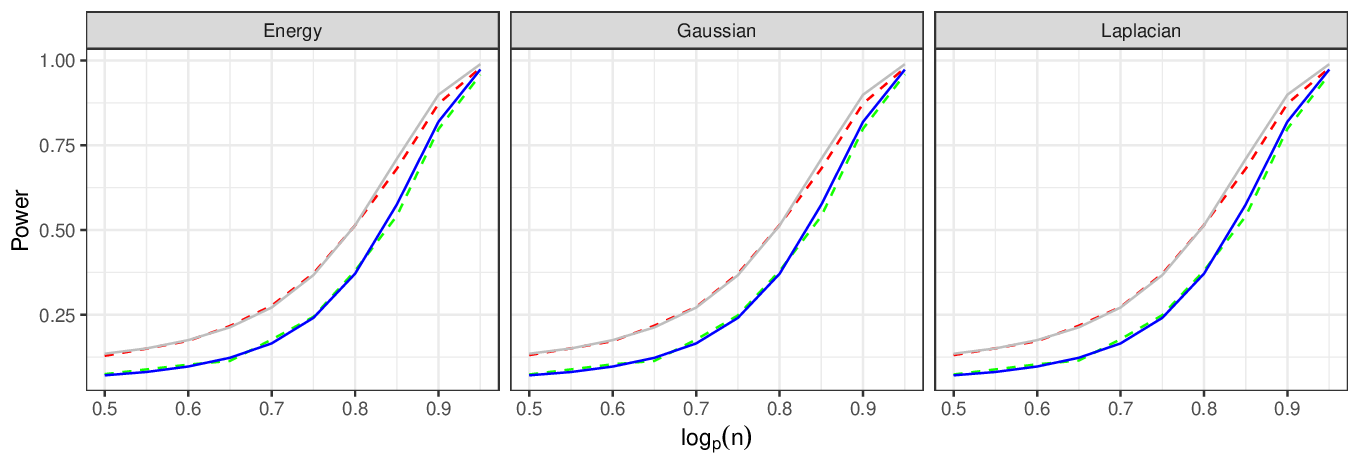}
    \caption{\footnotesize Empirical and theoretical powers for the MMD-based test, where the difference between the two distributions lies on the means. The results are obtained based on 1000 replications. 
    The dashed green, dashed red, solid blue and solid grey lines represent the empirical power curve with $\alpha=0.05$, empirical power curve with $\alpha=0.1$, theoretical power curve with $\alpha=0.05$ and theoretical power curve with $\alpha=0.1$, respectively.}
    \label{fig:nonnull1}
\end{figure}
\begin{figure}
    \centering
    \includegraphics[width=\textwidth]{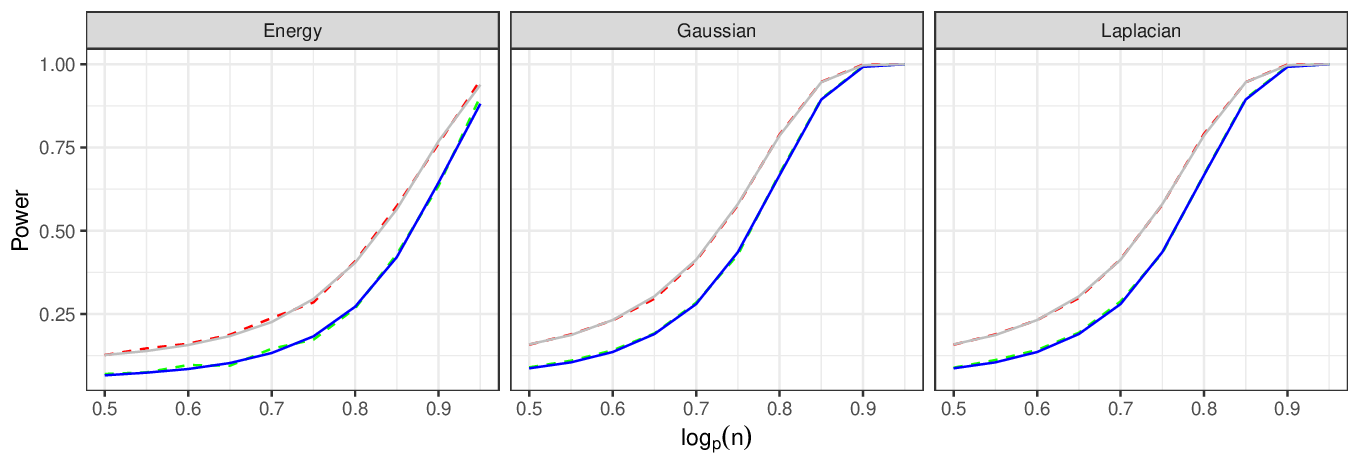}
    \caption{\footnotesize Empirical and theoretical powers for the MMD-based test, where the difference between the two distributions lies on the second-order moments. The results are obtained based on 1000 replications. 
    The dashed green, dashed red, solid blue and solid grey lines represent the empirical power curve with $\alpha=0.05$, empirical power curve with $\alpha=0.1$, theoretical power curve with $\alpha=0.05$ and theoretical power curve with $\alpha=0.1$, respectively.}
    \label{fig:nonnull2}
\end{figure}
We try $p=800$ and $n=p^{d}$, where $d$ ranges from 0.5 to 0.95 with the spacing 0.05. We plot the power of these tests when the significance level $\alpha=0.05$ or $0.1$. Figures \ref{fig:nonnull1} and \ref{fig:nonnull2} show that the empirical power is consistent with that predicted by our theory.

Finally, we validate the power behavior of the test in the case where the difference between the two distributions lies on the higher order moments. 
\begin{example}
$X_{i}(k)\stackrel{i.i.d.}{\sim}N(1,1)$ for $i=1,\ldots,n;k=1,\ldots,p$, and $Y_{j}(k)\stackrel{i.i.d.}{\sim}\text{Poisson}(1)$ for $j=1,\ldots,m;k=1,\ldots,p$. We try $p=50$, and $n=p^{d}$ where $d$ ranges from $1$ to $1.9$ with the spacing $0.1$. \end{example}
\begin{figure}
    \centering
    \includegraphics[width=\textwidth]{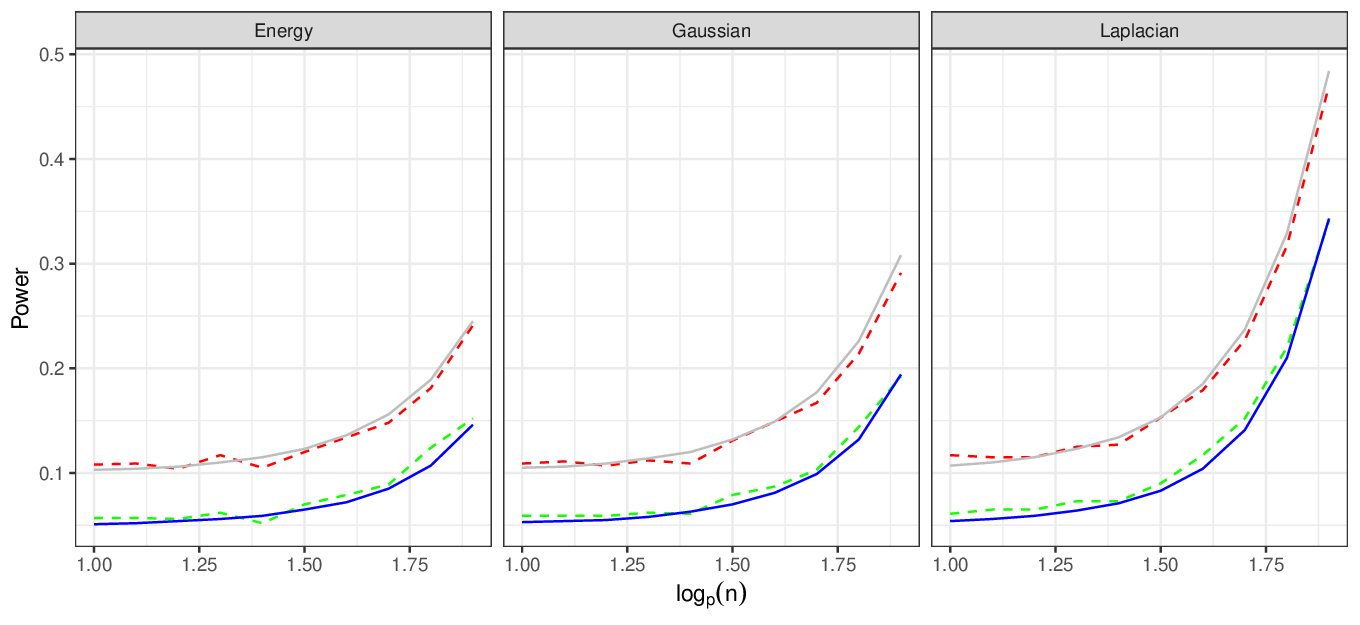}
    \caption{\footnotesize Empirical and theoretical powers for the MMD-based test, where the difference between the two distributions lies on the third-order moments. The results are obtained based on 1000 replications.
    The dashed green, dashed red, solid blue and solid grey lines represent the empirical power curve with $\alpha=0.05$, empirical power curve with $\alpha=0.1$, theoretical power curve with $\alpha=0.05$ and theoretical power curve with $\alpha=0.1$, respectively.}
    \label{fig:nonnull3}
\end{figure}
Since $N=o(p^2)$, the asymptotic power function is given in (\ref{power2}). As seen from Figure \ref{fig:nonnull3}, the empirical power curve matches with the asymptotic power, where we estimate the population MMD in the power function using Monte Carlo method.
Besides, we can verify that Assumptions \ref{assumption8}--\ref{assumption9} hold as $X$ and $Y$ have independent entries.
When $N= o(p^{3/2})$, the empirical power is close to the nominal level, which is consistent with the result in Corollary \ref{Coro-trivial-1}. A recent result in Proposition 3.4.3 of \citet{gao2021two} requires $p=o(N^{1/3})$ for the MMD-based test to have asymptotic power approaching one when the difference between the two distributions lies on the third moments. Our theory and numerical study suggest that the MMD-based test will have nontrivial power when $p=O(N^{2/3})$.

\section{Impact of kernel and bandwidth on power}\label{sec:impact-kernel}
In this section, we investigate the impact of kernel and bandwidth on the asymptotic and finite sample power behaviors in the regime $N=o(p)$. Our results shed some lights on how to select the optimal kernel that maximizes the power with respect to some specific alternatives. We consider the case where $\mu_1=\mu_2, \text{tr}(\Sigma_1)=\text{tr}(\Sigma_2)$ but $\Sigma_1\neq \Sigma_2$. In this case, we have $\tau=\tau_1=\tau_2=\tau_3$ and $\Delta_0=0$. By the expression of $T_1$ in (\ref{T1_2}), the asymptotic power depends on 
\[
\frac{T_1}{\surd{\text{var}({\color{red}\Delta_{1,1}})}}=\frac{f^{(2)}(\tau)}{|f^{(1)}(\tau)|}\frac{\gamma^{-1}\|\Sigma_1-\Sigma_2\|_{\rm F}^2}{\surd{\{\frac{2}{n(n-1)}\text{tr}(\Sigma_{1}^{2})+\frac{2}{m(m-1)}\text{tr}(\Sigma_{2}^{2})+\frac{4}{nm}\text{tr}(\Sigma_{1}\Sigma_{2})\}}}.
\]
In Example \ref{example3}, the bandwidth $\gamma=2p$ and $\tau=1$. Hence $f^{(2)}(\tau)/|f^{(1)}(\tau)|=1$ for both the Gaussian kernel and Laplace kernel, which explains why the asymptotic powers for the Gaussian kernel and Laplace kernel are the same in Fig. \ref{fig:nonnull2}. 

For fixed bandwidth, the optimal kernel is the one that maximizes $h_1(f)=f^{(2)}(\tau)/|f^{(1)}(\tau)|$. 
Define $h_2(\gamma)=f^{(2)}(2\text{tr}(\Sigma_1)/\gamma)\{|f^{(1)}(2\text{tr}(\Sigma_1)/\gamma)|\gamma\}^{-1}$. For fixed kernel, the bandwidth affects the power through $h_2(\gamma)$. For the energy distance, the bandwidth $\gamma$ only acts as a scaling factor, and thus does not affect the power. For other choices, the power increases when the bandwidth decreases; see Table \ref{tb4}. Below we provide some numerical evidence on how the choice of kernels and bandwidth affects the power, where we follow the setting in Example \ref{example3}.  

\begin{table}
\centering
\footnotesize
\def~{\hphantom{0}}
\caption{Impact of bandwidth on the asymptotic power for different kernels. }{%
\begin{tabular}{llll}
%\\
\\
Kernel & $f(x)$ & $h_2(\gamma)$, where $C=2\text{tr}(\Sigma_1)$ & Impact of bandwidth $\gamma$ on power\\
Gaussian & $\exp(-x)$ & $\gamma^{-1}$ & The power increases when $\gamma$ decreases \\
Laplace & $\exp(-\surd{x})$  & $\{(C\gamma)^{-1/2}+C^{-1}\}/2$ & The power increases when $\gamma$ decreases\\
Rational quadratic & $(1+x)^{-\alpha}$ & $(\alpha+1)/(\gamma+C)$ & The power increases when $\gamma$ decreases\\
Energy & $-\surd{x}$ & $(2C)^{-1}$ & The power is unrelated to $\gamma$
\end{tabular}}
\label{tb4}
\end{table}

\begin{example}
$X_i\stackrel{i.i.d.}{\sim}N(\mathbf{0}_{p\times 1},I_{p})$ for $i=1,\ldots,n$, and $Y_j\stackrel{i.i.d.}{\sim}N(\mathbf{0}_{p\times 1},\Sigma)$ for $j=1,\ldots,m$, where $\Sigma=(\sigma_{ij})_{i,j=1}^{p}$ with $\sigma_{ij}=0.5^{|i-j|}$. We try $p=800$, and $n=m=p^{d}$ where $d$ ranges from 0.5 to 0.95 with the spacing 0.05. We use the significance level $\alpha=0.05$ in the following two scenarios: 

1. Fixing the bandwidth $\gamma=p$, we compare the power of the MMD-based tests using the Gaussian, Laplace, rational quadratic (RQ) kernel (with $\alpha=1/2$) and energy distance. 

2. We compare the power of the MMD-based tests using the Gaussian kernel with different bandwidth $\gamma\in\{0.5p,p,1.5p,2p\}$. 
\end{example}

\begin{figure}
    \centering
    \includegraphics[width=\textwidth]{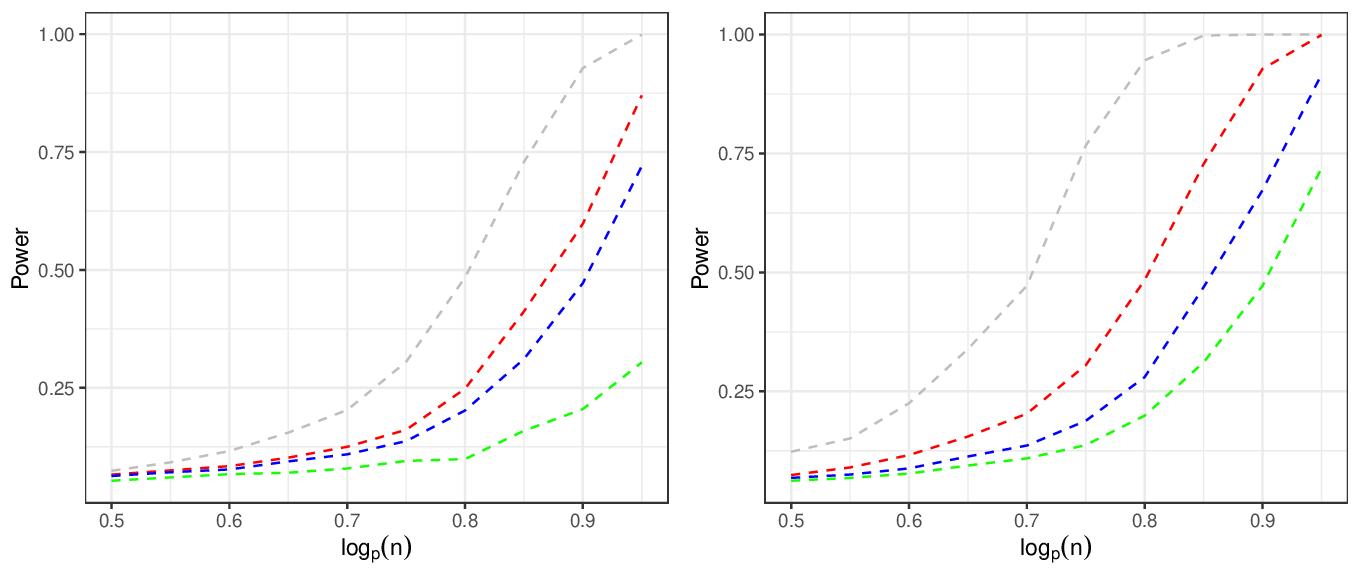}
    \caption{\footnotesize Impact of the kernel (left) and bandwidth (right) on the empirical power, where the results are based on 1000 replications. Left: the dashed grey, red, blue and green lines represent the Gaussian, Laplace, rational quadratic kernel and energy distance, respectively. Right: the dashed grey, red, blue and green lines represent the bandwidth $0.5p$, $p$, $1.5p$ and $2p$, respectively. }
    \label{fig:case2}
\end{figure}

In the first scenario, the bandwidth $\gamma=p$ is fixed and $\tau=2$. Simple algebra shows that
$h_1(\exp(-x))=1>h_1(\exp(-\surd{x}))=0.6>h_1((1+x)^{-1/2})=0.5>h_1(-\surd{x})=0.25$.
As seen from Figure \ref{fig:case2}, the MMD-based test using Gaussian kernel has the highest power, followed by the Laplace kernel and the rational quadratic kernel. The energy distance-based test has the lowest power. In the second scenario, the empirical power increases when the bandwidth decreases. The empirical observations from Fig. \ref{fig:case2} match with the theoretical results in Table \ref{tb4}. 

In Section \ref{secE} of the Appendix, we consider the case where $\Sigma_1=\Sigma_2$ but $\mu_1\neq \mu_2$. It is shown that the asymptotic power of the MMD-based test is not affected by the kernels and bandwidths, which is consistent with the finding in Fig. \ref{fig:nonnull1} associated with Example \ref{example2}.

\section{Discussion}
To conclude, we mention two future research directions. First, as Hilbert-Schmidt independence criterion (HSIC) can be viewed as MMD applied to the joint distribution between $X$ and $Y$ and the product of the two marginals, we expect similar results to hold for HSIC. A recent work along this direction is \citet{han2021generalized}. Second, it would be intriguing to study the high-dimensional behaviors of other popular discrepancy measures in the literature, such as the kernelized Stein discrepancy {\color{red}\citep{chwialkowski2016kernel,liu2016kernelized}} and the Wasserstein distance. 

\newpage

\begin{appendices}
\counterwithin{lemma}{section}
\counterwithin{figure}{section}
\counterwithin{example}{section}

\section{Some useful results}\label{secA}
\begin{lemma}\label{lemma-expansion}
Consider a function of the form $h(x)=g((a+x)^{1/2})$ for $a>0$ and $x\geq -a$, where $g$ is a real-valued function defined on $[0,+\infty)$. Suppose $$\sup_{1\leq s\leq l+1}\sup_{x\geq 0}|g^{(s)}(x)|<\infty.$$
Then we can write $h$ as
\begin{align*}
h(x)=\sum^{l}_{s=0}\frac{h^{(s)}(0)x^s}{s!}+c_{l+1,a}(x)x^{l+1}, \quad \sup_{x\geq -a}|c_{l+1,a}(x)|\leq C,    
\end{align*}
for some constant $C>0$ and any $x\geq -a$. The subscripts of the function $c(x)$ are used to indicate the dependency on $l+1$ and $a$. 
\end{lemma}
\begin{proof}
Our goal is to show for any $x\geq -a$
\begin{align*}
\bigg|h(x)-\sum^{l}_{s=0}\frac{h^{(s)}(0)x^s}{s!}\bigg|\leq C|x|^{l+1}    
\end{align*}
for some universal constant $C>0.$ First of all, some algebra yields that
\begin{align}\label{eq-der-h}
h^{(l+1)}(x)=\sum^{R}_{i=1}\frac{b_i(x)}{(a+x)^{c_i}},    
\end{align}
where $0\leq c_i\leq l+1/2$ for all $1\leq i\leq R,$ and $\sup_{1\leq i\leq R}\sup_{x\geq -a}|b_i(x)|<\infty$ under our assumption. 

When $x\geq 0$, the result holds by using the Taylor expansion with Lagrange remainder. We consider the case where $-a\leq x\leq 0.$ Using the Taylor expansion with integral remainder, we have
\begin{align*}
h(x)-\sum^{l}_{s=0}\frac{h^{(s)}(0)x^s}{s!}=\int^{x}_{0}\frac{h^{(l+1)}(t)(x-t)^l}{l!} dt=r(x).
\end{align*}
We show that $r(x)/x^{l+1}$ is well defined at the two boundary points $0$ and $-a$. By the L'Hospital's rule,
\begin{align*}
&\lim_{x\rightarrow 0^-}\bigg|\frac{r(x)}{x^{l+1}}\bigg|=\bigg|\frac{h^{(l+1)}(0)}{(l+1)!}\bigg|.
\end{align*}
On the other hand, when $x=-a,$ we have 
\begin{align*}
r(-a)=\frac{(-1)^{l+1}}{l!} \int^{0}_{-a}h^{(l+1)}(t)(a+t)^l dt.    
\end{align*}
By (\ref{eq-der-h}), it is not hard to verify that 
\begin{align*}
\bigg|\frac{r(-a)}{(-a)^{l+1}}\bigg|<\infty.    
\end{align*}
Therefore, the function $r(x)/x^{l+1}$ is continuous on the closed interval $[-a,0]$. By the extreme value theorem, there must exist a $C>0$ such that $|r(x)/x^{l+1}|\leq C.$ The conclusion thus follows by combining the results for $x\geq 0$ and $-a\leq x\leq 0$.
\end{proof}

\begin{lemma}
\label{lemmaa2}
Suppose $U$ is a random vector with independent entries $\{u_i\}_{i=1}^{q}$ such that $E(u_i)=0$ and $E(u_i^2)=1$. Then if $\max_{i}E(u_i^4)<\infty$ and $A_{q\times q}$ is a deterministic symmetric matrix, we have
\begin{align*}
\textup{var}(U^\top A U)=2\textup{tr}(A^2)+\sum_{i=1}^{q} \delta_i a_{ii}^2,
\end{align*}
where $\delta_i=E(u_i^4)-3$. 
\end{lemma}
\begin{proof}
We have
\begin{align*}
E(U^\top A U)=\sum_{i,j}a_{ij}E(u_iu_j)=\text{tr}(A),    
\end{align*}
and
\begin{align*}
E\{(U^\top A U)^2\}&=\sum_{i_1,j_1}\sum_{i_2,j_2}a_{i_1j_1}a_{i_2j_2}E(u_{i_1}u_{i_2}u_{j_1}u_{j_2})\\
&=\sum_i E(u_i^4)a_{ii}^2+2\sum_{i\neq j}a_{ij}^2+\sum_{i\neq j}a_{ii}a_{jj}
\\&=\sum_i\delta_i a_{ii}^2+2\text{tr}(A^2)+\{\text{tr}(A)\}^2.
\end{align*}
The result thus follows.
\end{proof}

\begin{lemma}
\label{lemmaa3}
Suppose $U$ is a random vector with independent entries $\{u_{i}\}_{i=1}^{q}$ such that $E(u_{i})=0$. Then if $\max_i E(|u_{i}|^{l})\leq \mu_l<\infty$ for $l\le L$ and $a_{q\times 1}$ is a deterministic column vector, we have
$$E\{(a^{\top}U)^{L}\}=O(\|a\|_{2}^{L}).$$
\end{lemma}
\begin{proof} Note that $\|a\|_p\leq \|a\|_q$ for $p\geq q\geq 1$, where $\|a\|_{p}=(\sum_{i}|a_{i}|^{p})^{1/p}$. We have
\begin{align*}
    |E\{(a^{\top}U)^{L}\}|&=\bigg|\sum_{i_1,\ldots,i_L}\prod_{l=1}^{L}a_{i_{l}}\times E\bigg(\prod_{l=1}^{L}u_{i_{l}}\bigg)\bigg|\le\sum_{\{l_{j}\}:l_{j}\ge 2,\sum_{j}l_{j}=L}C\prod_{j}\|a\|_{l_{j}}^{l_{j}}\mu_{l_{j}}\\
    &\leq C\sum_{\{l_{j}\}:l_{j}\ge 2,\sum_{j}l_{j}=L}\prod_{j}\|a\|_{2}^{l_{j}}\mu_{l_{j}}\\
    &\leq C\|a\|_{2}^{L}\times\sum_{\{l_{j}\}:l_{j}\ge 2,\sum_{j}l_{j}=L}\prod_{j}\mu_{l_{j}}=O(\|a\|_{2}^{L}),
\end{align*}
where $C$ is a constant related to $L$ only. 
\end{proof}

\begin{lemma}
\label{lemmaa4}
Suppose $U$ and $V$ are two independent but not necessarily identically distributed random vectors taking values in $\mathbb{R}^{q}$. The entries of $U$ and $V$ are also independent, and we denote the $i$-th entry of $U$ and $V$ by $u_{i}$ and $v_{i}$ respectively. Moreover, $E(u_{i})=E(v_{i})=0$, and $E(u_{i}^{2})=E(v_{i}^{2})=1$. Then if $\max_i E(|u_{i}|^{L})<\infty$ for $L\ge 0$, $\max_i E(|v_{i}|^{L})<\infty$ and $A_{q\times q}$ is a deterministic but not necessarily symmetric matrix, we have 
$$E\{(U^{\top}AV)^{L}\}=O(\{\textup{tr}(AA^{\top})\}^{L/2}).$$
\end{lemma}
\begin{proof}
Using Lemma \ref{lemmaa3} and independence of $U$ and $V$, 
\[
    {\color{red}|E\{(U^{\top}AV)^{L}\}|\le E|E\{(U^{\top}AV)^{L}\mid U\}|}\leq C\times E(\|A^{\top}U\|_{2}^{L}),
\]
where $C$ is a constant related to $L$ only. 

Using Lemma A.1 in \citet{bai1998no}, 
\[
    E(\|A^{\top}U\|_{2}^{L})=E\{(U^{\top}AA^{\top}U)^{L/2}\}=O(\{\textup{tr}(AA^{\top})\}^{L/2}).
\]
Note that the assumption that the entries of $U$ are identically distributed in Lemma A.1 in \citet{bai1998no} can be dropped by inspecting their proof. Also, $\text{tr}\{(AA^{\top})^{L/2}\}\leq \{\text{tr}(AA^{\top})\}^{L/2}$ as $AA^{\top}$ is positive semi-definite. 
\end{proof}

\begin{lemma}
\label{lemmaa5}
Suppose $U$ is a random vector with independent entries $\{u_i\}_{i=1}^{q}$ such that $E(u_i)=0$ and $E(u_i^2)=1$. Then if $\max_i E(u_{i}^{2L})<\infty$ and $A_{q\times q}$ is a deterministic but not necessarily symmetric matrix, we have
$$E[\{U^{\top}AU-\textup{tr}(A)\}^{L}]=O(\{\textup{tr}(AA^{\top})\}^{L/2}).$$
\end{lemma}
\begin{proof}
See Lemma 2.7 in \citet{bai1998no} and its proof. Again, the assumption that $\{u_i\}_{i=1}^{q}$ are identically distributed in Lemma 2.7 in \citet{bai1998no} can be dropped. 
\end{proof}

\section{Technical details}\label{secB}
\subsection{Taylor expansion for the kernel function}\label{secB.1}
Consider $k(X_{i_1},X_{i_2})=f(p^{-1}\|X_{i_1}-X_{i_2}\|_{2}^2)$ for illustration. The Taylor expansions of $f(p^{-1}\|Y_{j_1}-Y_{j_2}\|_{2}^2)$ and $f(p^{-1}\|X_{i}-Y_{j}\|_{2}^2)$ follow the same line. 

Let $g(x)=f(x^2)$ as defined in Assumption \ref{assumption5} and write $\widetilde{X}_{i_1,i_2}=p^{-1}\|X_{i_{1}}-X_{i_{2}}\|_{2}^{2}-\tau_1$. Under the assumption $\sup_{1\leq s\leq l+1}\sup_{x\geq 0}|g^{(s)}(x)|<\infty$, we can apply Lemma \ref{lemma-expansion} with $x=\widetilde{X}_{i_1,i_2}$, $a=\tau_1$ and $h(x)=g((a+x)^{1/2})$ to obtain
\begin{align}\label{eq-expan-h}
h(\widetilde{X}_{i_1,i_2})=\sum^{l}_{s=0}\frac{h^{(s)}(0)\widetilde{X}_{i_1,i_2}^s}{s!}+c_{l+1,\tau_1}(\widetilde{X}_{i_1,i_2})\widetilde{X}_{i_1,i_2}^{l+1},
\end{align}
where $c_{l+1,\tau_1}(\cdot)$ is a bounded function. 

Note that $f$ is related to $h$ through $f(x+a)=h(x)$ and hence $f^{(s)}(x+a)=h^{(s)}(x)$ for any $s\geq 0$. Thus (\ref{eq-expan-h}) implies that
\begin{align*}
f(p^{-1}\|X_{i_1}-X_{i_2}\|_{2}^2)=f(\widetilde{X}_{i_1,i_2}+\tau_1)=\sum^{l}_{s=0}\frac{f^{(s)}(\tau_1)\widetilde{X}_{i_1,i_2}^s}{s!}+c_{l+1,\tau_1}(\widetilde{X}_{i_1,i_2})\widetilde{X}_{i_1,i_2}^{l+1}.
\end{align*}

\subsection{The variance and the mean of the second-order term in the non-null setting}
\label{secB.2}
We first handle the variance of $\widetilde{\Delta}_{2}$. Define $P_{1}=\{n(n-1)\}^{-1}\sum_{i_{1}\neq i_{2}}c_{2,\tau_1}(\widetilde{X}_{i_1,i_2})(\|X_{i_{1}}-X_{i_{2}}\|_{2}^{2}-p\tau_{1})^{2}$, $P_{2}=\{m(m-1)\}^{-1}\sum_{j_{1}\neq j_{2}}c_{2,\tau_2}(\widetilde{Y}_{j_1,j_2})(\|Y_{j_{1}}-Y_{j_{2}}\|_{2}^{2}-p\tau_{2})^{2}$ and $P_{3}=-2(nm)^{-1}\sum_{i,j}c_{2,\tau_3}(\widetilde{Z}_{i,j})(\|X_{i}-Y_{j}\|_{2}^{2}-p\tau_{3})^{2}$. We note that
\begin{align*}
    \text{var}(P_{1})&=\frac{1}{n^{2}(n-1)^{2}}\sum_{i_{1}\neq i_{2}}\sum_{i_{3}\neq i_{4}}\text{cov}\{c_{2,\tau_1}(\widetilde{X}_{i_1,i_2})(\|X_{i_{1}}-X_{i_{2}}\|_{2}^{2}-p\tau_{1})^{2},c_{2,\tau_1}(\widetilde{X}_{i_3,i_4})(\|X_{i_{3}}-X_{i_{4}}\|_{2}^{2}-p\tau_{1})^{2}\}\\
    &=\frac{4}{n^{2}(n-1)^{2}}\sum_{i_{1}\neq i_{2}}\sum_{i_{3}\neq i_{1},i_{2}}\text{cov}\{c_{2,\tau_1}(\widetilde{X}_{i_1,i_2})(\|X_{i_{1}}-X_{i_{2}}\|_{2}^{2}-p\tau_{1})^{2},c_{2,\tau_1}(\widetilde{X}_{i_1,i_3})(\|X_{i_{1}}-X_{i_{3}}\|_{2}^{2}-p\tau_{1})^{2}\}\\
    &\quad+\frac{2}{n^{2}(n-1)^{2}}\sum_{i_{1}\neq i_{2}}\text{var}\{c_{2,\tau_1}(\widetilde{X}_{i_1,i_2})(\|X_{i_{1}}-X_{i_{2}}\|_{2}^{2}-p\tau_{1})^{2}\}\\
    &=\frac{4(n-2)}{n(n-1)}\text{cov}\{c_{2,\tau_1}(\widetilde{X}_{1,2})(\|X_{1}-X_{2}\|_{2}^{2}-p\tau_{1})^{2},c_{2,\tau_1}(\widetilde{X}_{1,3})(\|X_{1}-X_{3}\|_{2}^{2}-p\tau_{1})^{2}\}\\
    &\quad+\frac{2}{n(n-1)}\text{var}\{c_{2,\tau_1}(\widetilde{X}_{1,2})(\|X_{1}-X_{2}\|_{2}^{2}-p\tau_{1})^{2}\}, 
\end{align*}
where $X_{1},X_{2},X_{3}$ are i.i.d copies of $X$. Furthermore, we have
\begin{align*}
    \text{var}(P_{1})&\le\frac{4(n-2)}{n(n-1)}\surd{\text{var}\{c_{2,\tau_1}(\widetilde{X}_{1,2})(\|X_{1}-X_{2}\|_{2}^{2}-p\tau_{1})^{2}\}}\times\surd{\text{var}\{c_{2,\tau_1}(\widetilde{X}_{1,3})(\|X_{1}-X_{3}\|_{2}^{2}-p\tau_{1})^{2}\}}\\
    &\quad+\frac{2}{n(n-1)}\text{var}\{c_{2,\tau_1}(\widetilde{X}_{1,2})(\|X_{1}-X_{2}\|_{2}^{2}-p\tau_{1})^{2}\}\\
    &=\frac{2(2n-3)}{n(n-1)}\text{var}\{c_{2,\tau_1}(\widetilde{X}_{1,2})(\|X_{1}-X_{2}\|_{2}^{2}-p\tau_{1})^{2}\}\\
    &\le\frac{2(2n-3)}{n(n-1)}E\{c_{2,\tau_1}^{2}(\widetilde{X}_{1,2})(\|X_{1}-X_{2}\|_{2}^{2}-p\tau_{1})^{4}\}\\
    &=O(N^{-1})\times E\{(\|X_{1}-X_{2}\|_{2}^{2}-p\tau_{1})^{4}\},
\end{align*}
since $c_{2,\tau_1}(\cdot)$ is bounded. Hence we need to study the term $E\{(\|X_{1}-X_{2}\|_{2}^{2}-p\tau_{1})^{4}\}$. We have
\begin{align*}
    E\{(\|X_{1}-X_{2}\|_{2}^{2}-p\tau_{1})^{4}\}&=E[\{\|X_{1}-\mu_{1}\|_{2}^{2}+\|X_{2}-\mu_{1}\|_{2}^{2}-2(X_{1}-\mu_{1})^{\top}(X_{2}-\mu_{1})-2\text{tr}(\Sigma_{1})\}^{4}]\\
    &=E[\{U_{1}^{\top}\Gamma_{1}^{\top}\Gamma_{1}U_{1}-\text{tr}(\Sigma_{1})+U_{2}^{\top}\Gamma_{1}^{\top}\Gamma_{1}U_{2}-\text{tr}(\Sigma_{1})-2U_{1}^{\top}\Gamma_{1}^{\top}\Gamma_{1}U_{2}\}^{4}]\\
    &\leq C\times[E\{(U_{1}^{\top}\Gamma_{1}^{\top}\Gamma_{1}U_{1}-\text{tr}(\Sigma_{1}))^{4}\}+E\{(U_{1}^{\top}\Gamma_{1}^{\top}\Gamma_{1}U_{2})^{4}\}],
\end{align*}
for some constant $C>0$. Using Lemma \ref{lemmaa5}, we know that
\[
E\{(U_{1}^{\top}\Gamma_{1}^{\top}\Gamma_{1}U_{1}-\text{tr}(\Sigma_{1}))^{4}\}=O(\{\text{tr}(\Gamma_{1}^{\top}\Gamma_{1}\Gamma_{1}^{\top}\Gamma_{1})\}^{2})=O(\text{tr}^{2}(\Sigma_{1}^{2}))=O(p^{2}).
\]
Using Lemma \ref{lemmaa4}, 
\[
E\{(U_{1}^{\top}\Gamma_{1}^{\top}\Gamma_{1}U_{2})^{4}\}=O(\text{tr}^{2}(\Sigma_{1}^{2}))=O(p^{2}).
\]
Therefore, 
$$\text{var}(P_{1})=O(N^{-1})\times E\{(\|X_{1}-X_{2}\|_{2}^{2}-p\tau_{1})^{4}\}=O(N^{-1} p^{2}).$$
Similarly, 
\begin{align*}
    \text{var}(P_{2})&=\frac{4(m-2)}{m(m-1)}\text{cov}\{c_{2,\tau_2}(\widetilde{Y}_{1,2})(\|Y_{1}-Y_{2}\|_{2}^{2}-p\tau_{2})^{2},c_{2,\tau_2}(\widetilde{Y}_{1,3})(\|Y_{1}-Y_{3}\|_{2}^{2}-p\tau_{2})^{2}\}\\
    &\quad+\frac{2}{m(m-1)}\text{var}\{c_{2,\tau_2}(\widetilde{Y}_{1,2})(\|Y_{1}-Y_{2}\|_{2}^{2}-p\tau_{2})^{2}\}\\
    &=O(N^{-1})\times E\{(\|Y_{1}-Y_{2}\|_{2}^{2}-p\tau_{2})^{4}\}=O(N^{-1} p^{2}),
\end{align*}
and
\begin{align*}
    \text{var}(P_{3})&=\frac{4}{n^{2}m^{2}}\sum_{i_{1}}\sum_{i_{2}}\sum_{j_{1}}\sum_{j_{2}}\text{cov}\{c_{2,\tau_3}(\widetilde{Z}_{i_1,j_1})(\|X_{i_1}-Y_{j_1}\|_{2}^{2}-p\tau_{3})^{2},c_{2,\tau_3}(\widetilde{Z}_{i_2,j_2})(\|X_{i_2}-Y_{j_2}\|_{2}^{2}-p\tau_{3})^{2}\}\\
    &=\frac{4}{n^{2}m^{2}}\sum_{i}\sum_{j_{1}\neq j_{2}}\text{cov}\{c_{2,\tau_3}(\widetilde{Z}_{i,j_1})(\|X_{i}-Y_{j_1}\|_{2}^{2}-p\tau_{3})^{2},c_{2,\tau_3}(\widetilde{Z}_{i,j_2})(\|X_{i}-Y_{j_2}\|_{2}^{2}-p\tau_{3})^{2}\}\\
    &\quad+\frac{4}{n^{2}m^{2}}\sum_{i_{1}\neq i_{2}}\sum_{j}\text{cov}\{c_{2,\tau_3}(\widetilde{Z}_{i_1,j})(\|X_{i_1}-Y_{j}\|_{2}^{2}-p\tau_{3})^{2},c_{2,\tau_3}(\widetilde{Z}_{i_2,j})(\|X_{i_2}-Y_{j}\|_{2}^{2}-p\tau_{3})^{2}\}\\
    &\quad+\frac{4}{n^{2}m^{2}}\sum_{i}\sum_{j}\text{var}\{c_{2,\tau_3}(\widetilde{Z}_{i,j})(\|X_{i}-Y_{j}\|_{2}^{2}-p\tau_{3})^{2}\}\\
    &=\frac{4(m-1)}{nm}\text{cov}\{c_{2,\tau_3}(\widetilde{Z}_{1,1})(\|X_{1}-Y_{2}\|_{2}^{2}-p\tau_{3})^{2},c_{2,\tau_3}(\widetilde{Z}_{1,2})(\|X_{1}-Y_{2}\|_{2}^{2}-p\tau_{3})^{2}\}\\
    &\quad+\frac{4(n-1)}{nm}\text{cov}\{c_{2,\tau_3}(\widetilde{Z}_{1,1})(\|X_{1}-Y_{1}\|_{2}^{2}-p\tau_{3})^{2},c_{2,\tau_3}(\widetilde{Z}_{2,1})(\|X_{2}-Y_{1}\|_{2}^{2}-p\tau_{3})^{2}\}\\
    &\quad+\frac{4}{nm}\text{var}\{c_{2,\tau_3}(\widetilde{Z}_{1,1})(\|X_{1}-Y_{1}\|_{2}^{2}-p\tau_{3})^{2}\}\\
    &=O(N^{-1})\times E\{(\|X_{1}-Y_{1}\|_{2}^{2}-p\tau_{3})^{4}\}\\
    &=O(N^{-1})\times(E\{(\|X_{1}-\mu_{1}\|_{2}^{2}-\text{tr}(\Sigma_{1}))^{4}\}+E\{(\|Y_{1}-\mu_{2}\|_{2}^{2}-\text{tr}(\Sigma_{2}))^{4}\}\\
    &\quad+E\{((X_{1}-\mu_{1})^{\top}(Y_{1}-\mu_{2}))^{4}\}+E\{((\mu_{1}-\mu_{2})^{\top}(X_{1}-\mu_{1}))^{4}\}+E\{((\mu_{2}-\mu_{1})^{\top}(Y_{1}-\mu_{2}))^{4}\}].
\end{align*}
As shown in the calculations of $\text{var}(P_1)$ and $\text{var}(P_2)$, we have $E\{(\|X_{1}-\mu_{1}\|_{2}^{2}-\text{tr}(\Sigma_{1}))^{4}\}=O(p^2)$ and $E\{(\|Y_{1}-\mu_{2}\|_{2}^{2}-\text{tr}(\Sigma_{2}))^{4}\}=O(p^2)$. Thus, 
\begin{align*}
    \text{var}(P_{3})&=O(N^{-1})\times[O(p^{2})+E\{(U_{1}^{\top}\Gamma_{1}^{\top}\Gamma_{2}V_{1})^{4}\}+E\{((\mu_{1}-\mu_{2})^{\top}\Gamma_{1}U_{1})^{4}\}+E\{((\mu_{2}-\mu_{1})^{\top}\Gamma_{2}V_{1})^{4}\}]\\
    &=O(N^{-1})\times[O(p^{2})+O(\text{tr}^{2}(\Sigma_{1}\Sigma_{2}))+O(\{(\mu_{1}-\mu_{2})^{\top}\Sigma_{1}(\mu_{1}-\mu_{2})\}^{2})\\
    &\quad+O(\{(\mu_{2}-\mu_{1})^{\top}\Sigma_{2}(\mu_{2}-\mu_{1})\}^{2})]\\
    &=O(N^{-1} p^{2}),
\end{align*}
where we have used Lemma \ref{lemmaa3}, Lemma \ref{lemmaa4} and Assumption \ref{assumption3}. Putting together the above results, we have
$$\text{var}(\widetilde{\Delta}_{2})=p^{-4}\times\text{var}(P_{1}+P_{2}+P_{3})\leq 3p^{-4}\times\{\text{var}(P_{1})+\text{var}(P_{2})+\text{var}(P_{3})\}=O(N^{-1} p^{-2}).$$
under Assumptions \ref{assumption1}--\ref{assumption5}. 

Now we cope with the mean of $\widetilde{\Delta}_{2}$. Under Assumption \ref{assumption5}, we can also perform a third-order Taylor expansion (again see Section \ref{secB.1}):
$$f(p^{-1}\|X_{i_1}-X_{i_2}\|_{2}^2)=f(\tau_1)+f^{(1)}(\tau_1)\widetilde{X}_{i_1,i_2}+\frac{f^{(2)}(\tau_1)}{2}\widetilde{X}_{i_1,i_2}^{2}+c_{3,\tau_1}(\widetilde{X}_{i_1,i_2})\widetilde{X}_{i_1,i_2}^{3},$$
where $c_{3,\tau_1}(\cdot)$ is some bounded function. Comparing with the second-order expansion, we immediately have
$$c_{2,\tau_1}(\widetilde{X}_{i_1,i_2})=\frac{1}{2}f^{(2)}(\tau_1)+c_{3,\tau_1}(\widetilde{X}_{i_1,i_2})\widetilde{X}_{i_1,i_2}.$$
Similarly,
\begin{align*}
    c_{2,\tau_2}(\widetilde{Y}_{j_1,j_2})&=\frac{1}{2}f^{(2)}(\tau_2)+c_{3,\tau_2}(\widetilde{Y}_{j_1,j_2})\widetilde{Y}_{j_1,j_2},\\
    c_{2,\tau_3}(\widetilde{Z}_{i,j})&=\frac{1}{2}f^{(2)}(\tau_3)+c_{3,\tau_3}(\widetilde{Z}_{i,j})\widetilde{Z}_{i,j}.
\end{align*}
Using the definition of $T_{1}$, we obtain
\begin{align*}
    &\quad|E(\widetilde{\Delta}_{2})-T_{1}|\\
    &\le \frac{1}{p^2}\times\big[E\{|c_{2,\tau_1}(\widetilde{X}_{1,2})-\frac{1}{2}f^{(2)}(\tau_1)|\times(\|X_{1}-X_{2}\|_{2}^{2}-p\tau_{1})^{2}\}\\
    &\quad+E\{|c_{2,\tau_2}(\widetilde{Y}_{1,2})-\frac{1}{2}f^{(2)}(\tau_2)|\times(\|Y_{1}-Y_{2}\|_{2}^{2}-p\tau_{2})^{2}\}\\
    &\quad+E\{|2c_{2,\tau_3}(\widetilde{Z}_{1,1})-f^{(2)}(\tau_3)|\times(\|X_{1}-Y_{1}\|_{2}^{2}-p\tau_{3})^{2}\}\big]\\
    &=\frac{1}{p^2}\times\big[E\{|c_{3,\tau_1}(\widetilde{X}_{1,2})\widetilde{X}_{1,2}|\times(\|X_{1}-X_{2}\|_{2}^{2}-p\tau_{1})^{2}\}\\
    &\quad+E\{|c_{3,\tau_2}(\widetilde{Y}_{1,2})\widetilde{Y}_{1,2}|\times(\|Y_{1}-Y_{2}\|_{2}^{2}-p\tau_{2})^{2}\}+2E\{|c_{3,\tau_3}(\widetilde{Z}_{1,1})\widetilde{Z}_{1,1}|\times(\|X_{1}-Y_{1}\|_{2}^{2}-p\tau_{3})^{2}\}\big]\\
    &\leq \frac{C}{p^3}\times\big\{E(|\|X_{1}-X_{2}\|_{2}^{2}-p\tau_{1}|^{3})+E(|\|Y_{1}-Y_{2}\|_{2}^{2}-p\tau_{2}|^{3})+2E(|\|X_{1}-Y_{1}\|_{2}^{2}-p\tau_{3}|^{3})\big\},
\end{align*}
for some constant $C>0$. Following the previous calculations of the variance, we have
\begin{align*}
    E(|\|X_{1}-X_{2}\|_{2}^{2}-p\tau_{1}|^{3})&\le\surd{E\{(\|X_{1}-X_{2}\|_{2}^{2}-p\tau_{1})^{2}\}}\times \surd{E\{(\|X_{1}-X_{2}\|_{2}^{2}-p\tau_{1})^{4}\}}\\
    &=\surd{O(p)}\times \surd{O(p^{2})}=O(p^{\frac{3}{2}}).
\end{align*}
Similarly, $E(|\|Y_{1}-Y_{2}\|_{2}^{2}-p\tau_{2}|^{3})=O(p^{3/2})$ and $E(|\|X_{1}-Y_{1}\|_{2}^{2}-p\tau_{3}|^{3})=O(p^{3/2})$. Therefore, 
$$E(\widetilde{\Delta}_{2})-T_{1}=O(p^{-\frac{3}{2}}).$$

\subsection{The variance of the second-order term under the null}\label{secB.3}
Under the null $H_0:P_X=P_Y$, let $\tau_{1}=\tau_{2}=\tau_{3}=\tau$ and observe that
\begin{align*}
    &\quad\text{cov}\{c_{2,\tau}(\widetilde{X}_{1,2})(\|X_{1}-X_{2}\|_{2}^{2}-p\tau)^{2},c_{2,\tau}(\widetilde{X}_{1,3})(\|X_{1}-X_{3}\|_{2}^{2}-p\tau)^{2}\}
    \\&=\text{cov}\{c_{2,\tau}(\widetilde{Y}_{1,2})(\|Y_{1}-Y_{2}\|_{2}^{2}-p\tau)^{2},c_{2,\tau}(\widetilde{Y}_{1,3})(\|Y_{1}-Y_{3}\|_{2}^{2}-p\tau)^{2}\}\\
    &=\text{cov}\{c_{2,\tau}(\widetilde{Z}_{1,1})(\|X_{1}-Y_{1}\|_{2}^{2}-p\tau)^{2},c_{2,\tau}(\widetilde{Z}_{1,2})(\|X_{1}-Y_{2}\|_{2}^{2}-p\tau)^{2}\}
    \\&=\text{cov}\{c_{2,\tau}(\widetilde{Z}_{1,1})(\|X_{1}-Y_{1}\|_{2}^{2}-p\tau)^{2},c_{2,\tau}(\widetilde{Z}_{2,1})(\|X_{2}-Y_{1}\|_{2}^{2}-p\tau)^{2}\}\\
    &=\text{cov}\{c_{2,\tau}(\widetilde{X}_{1,2})(\|X_{1}-X_{2}\|_{2}^{2}-p\tau)^{2},c_{2,\tau}(\widetilde{Z}_{1,1})(\|X_{1}-Y_{1}\|_{2}^{2}-p\tau)^{2}\}
    \\&=\text{cov}\{c_{2,\tau}(\widetilde{Y}_{1,2})(\|Y_{1}-Y_{2}\|_{2}^{2}-p\tau)^{2},c_{2,\tau}(\widetilde{Z}_{1,1})(\|X_{1}-Y_{1}\|_{2}^{2}-p\tau)^{2}\}=Q_{1},
\end{align*}
and
\begin{align*}
    &\quad\text{var}\{c_{2,\tau}(\widetilde{X}_{1,2})(\|X_{1}-X_{2}\|_{2}^{2}-p\tau)^{2}\}\\
    &=\text{var}\{c_{2,\tau}(\widetilde{Y}_{1,2})(\|Y_{1}-Y_{2}\|_{2}^{2}-p\tau)^{2}\}\\
    &=\text{var}\{c_{2,\tau}(\widetilde{Z}_{1,1})(\|X_{1}-Y_{1}\|_{2}^{2}-p\tau)^{2}\}=Q_{2}.
\end{align*}
Then along with the calculations in Section \ref{secB.2}, some algebra yields that
\begin{align*}
    \text{var}(P_{1})&=\frac{4(n-2)}{n(n-1)}Q_{1}+\frac{2}{n(n-1)}Q_{2},\\
    \text{var}(P_{2})&=\frac{4(m-2)}{m(m-1)}Q_{1}+\frac{2}{m(m-1)}Q_{2},\\
    \text{var}(P_{3})&=\frac{4(n+m-2)}{nm}Q_{1}+\frac{4}{nm}Q_{2}.
\end{align*}
Because the two samples are independent, $\text{cov}(P_{1},P_{2})=0$. Also, under $H_{0}$, we have
\begin{align*}
    &\text{cov}(P_{1},P_{3})\\
    =&-\frac{2}{n^{2}(n-1)m}\sum_{i_{1}\neq i_{2}}\sum_{i_{3}}\sum_{j}\text{cov}\{c_{2,\tau}(\widetilde{X}_{i_1,i_2})(\|X_{i_{1}}-X_{i_{2}}\|_{2}^{2}-p\tau)^{2},c_{2,\tau}(\widetilde{Z}_{i_3,j})(\|X_{i_3}-Y_{j}\|_{2}^{2}-p\tau)^{2}\}\\
    =&-\frac{4}{n^{2}(n-1)m}\sum_{i_{1}\neq i_{2}}\sum_{j}\text{cov}\{c_{2,\tau}(\widetilde{X}_{i_1,i_2})(\|X_{i_{1}}-X_{i_{2}}\|_{2}^{2}-p\tau)^{2},c_{2,\tau}(\widetilde{Z}_{i_1,j})(\|X_{i_1}-Y_{j}\|_{2}^{2}-p\tau)^{2}\}\\
    =&-\frac{4}{n}Q_{1}.\\
\end{align*}
Similarly, 
$$\text{cov}(P_{2},P_{3})=-\frac{4}{m}Q_{1}.$$
In summary, under $H_{0}$,
\begin{align*}
    \text{var}(\widetilde{\Delta}_{2})&=p^{-4}\times\text{var}(P_{1}+P_{2}+P_{3})=\frac{2}{p^{4}}\times\bigg\{\frac{1}{n(n-1)}+\frac{1}{m(m-1)}+\frac{2}{nm}\bigg\}(Q_{2}-2Q_{1})\\
    &\le\frac{6}{p^{4}}\times\left\{\frac{1}{n(n-1)}+\frac{1}{m(m-1)}+\frac{2}{nm}\right\}Q_{2},
\end{align*}
where we have used the Cauchy–Schwarz inequality, i.e., $Q_{1}\ge-Q_{2}$. 

As shown in Section \ref{secB.2}, $Q_{2}=O(p^{2})$ under Assumptions \ref{assumption1}--\ref{assumption5}, thus, $\text{var}(\widetilde{\Delta}_{2})=O(N^{-2} p^{-2})$, the order of which is smaller than the one in the non-null setting.

\subsection{The variances of $\Delta_{s}$ for $s=2,\ldots,l-1,$ and $\widetilde{\Delta}_{l}$ in Theorem \ref{theorem3.9}}\label{secB.4}
We first handle the variance of $\Delta_{s}\ (s=2,\ldots,l-1)$ under Assumption \ref{assumption8}. We do not require the first $(l-1)$ moments of $X$ and $Y$ to be equal for now. We aim to show that $\text{var}(\Delta_{s})=o(N^{-2} p^{-1})$. Recall that 
\begin{align*}
    \Delta_{s}-E(\Delta_{s})&=\frac{1}{s!}\bigg[\frac{f^{(s)}(\tau_1)}{n(n-1)}\sum_{i_{1}\neq i_{2}}\{\widetilde{X}_{i_1,i_2}^s-E(\widetilde{X}_{i_1,i_2}^s)\}+\frac{f^{(s)}(\tau_2)}{m(m-1)}\sum_{j_{1}\neq j_{2}}\{\widetilde{Y}_{j_1,j_2}^s-E(\widetilde{Y}_{j_1,j_2}^s)\}
    \\&\quad-\frac{2}{nm}f^{(s)}(\tau_3)\sum_{i,j}\{\widetilde{Z}_{i,j}^s-E(\widetilde{Z}_{i,j}^s)\}\bigg].
\end{align*}
We can express the term inside the square bracket as
\begin{align*}
&\quad\frac{f^{(s)}(\tau_1)}{n(n-1)}\sum_{i_{1}\neq i_{2}}\{\widetilde{X}_{i_1,i_2}^s-E(\widetilde{X}_{i_1,i_2}^s)\}+\frac{f^{(s)}(\tau_2)}{m(m-1)}\sum_{j_{1}\neq j_{2}}\{\widetilde{Y}_{j_1,j_2}^s-E(\widetilde{Y}_{j_1,j_2}^s)\}-\frac{2}{nm}f^{(s)}(\tau_3)\sum_{i,j}\{\widetilde{Z}_{i,j}^s-E(\widetilde{Z}_{i,j}^s)\}
\\&=\frac{f^{(s)}(\tau_1)}{n(n-1)}\sum_{i_{1}\neq i_{2}}\{\widetilde{X}_{i_1,i_2}^s-E(\widetilde{X}_{i_1,i_2}^s\mid X_{i_1})-E(\widetilde{X}_{i_1,i_2}^s\mid X_{i_2})+E(\widetilde{X}_{i_1,i_2}^s)\}
\\&\quad+\frac{f^{(s)}(\tau_2)}{m(m-1)}\sum_{j_{1}\neq j_{2}}\{\widetilde{Y}_{j_1,j_2}^s-E(\widetilde{Y}_{j_1,j_2}^s\mid Y_{j_1})-E(\widetilde{Y}_{j_1,j_2}^s\mid Y_{j_2})+E(\widetilde{Y}_{j_1,j_2}^s)\}
\\&\quad-\frac{2}{nm}f^{(s)}(\tau_3)\sum_{i,j}\{\widetilde{Z}_{i,j}^s-E(\widetilde{Z}_{i,j}^s\mid X_i)-E(\widetilde{Z}_{i,j}^s\mid Y_j)+E(\widetilde{Z}_{i,j}^s)\}
\\&\quad+\frac{2}{n}\sum_{i}\{f^{(s)}(\tau_1)E(\widetilde{X}_{i,0}^s\mid X_i)-f^{(s)}(\tau_1)E(\widetilde{X}_{i,0}^s)-f^{(s)}(\tau_3)E(\widetilde{Z}_{i,0}^s\mid X_i)+f^{(s)}(\tau_3)E(\widetilde{Z}_{i,0}^s)\}
\\&\quad+\frac{2}{m}\sum_{j}\{f^{(s)}(\tau_2)E(\widetilde{Y}_{j,0}^s\mid Y_j)-f^{(s)}(\tau_2)E(\widetilde{Y}_{j,0}^s)-f^{(s)}(\tau_3)E(\widetilde{Z}_{0,j}^s\mid Y_j)+f^{(s)}(\tau_3)E(\widetilde{Z}_{0,j}^s)\}\\
&:=I_1+I_2+I_3
\end{align*}
where the sum of the first three terms is denoted by $I_1$, and $(X_0,Y_{0})$ is an i.i.d. copy of $(X,Y)$. 

Using the double-centering property, we have 
$E(I_1)=0$. Also, following Section \ref{secB.2} and using Lemmas \ref{lemmaa3}-\ref{lemmaa5}, 
\begin{align*}
    \text{var}(I_1)&=O(N^{-2})\times\{\text{var}(\widetilde{X}_{1,2}^s)+\text{var}(\widetilde{Y}_{1,2}^s)+\text{var}(\widetilde{Z}_{1,1}^s)\}\\
    &=O(N^{-2} p^{-2s})\times[E\{(\|X_{1}-X_{2}\|_{2}^{2}-p\tau_1)^{2s}\}+E\{(\|Y_{1}-Y_{2}\|_{2}^{2}-p\tau_2)^{2s}\}\\
    &\quad+E\{(\|X_{1}-Y_{1}\|_{2}^{2}-p\tau_3)^{2s}\}]\\
    &=O(N^{-2} p^{-2s})\times[O(\text{tr}^{s}(\Sigma_{1}^2))+O(\text{tr}^{s}(\Sigma_{2}^2))+O(\text{tr}^{s}(\Sigma_{1}\Sigma_{2}))\\
    &\quad+O(\{(\mu_{1}-\mu_{2})^{\top}\Sigma_{1}(\mu_{1}-\mu_{2})\}^{s})+O(\{(\mu_{2}-\mu_{1})^{\top}\Sigma_{2}(\mu_{2}-\mu_{1})\}^{s})]\\
    &=O(N^{-2} p^{-s})=o(N^{-2} p^{-1}).
\end{align*}
Thus, it remains to analyze $I_2$ (the analysis of $I_3$ is the same). Note that
\begin{align*}
\text{var}(I_2)=\frac{4}{n}\text{var}\{f^{(s)}(\tau_1)E(\widetilde{X}_{1,0}^s\mid X_1)-f^{(s)}(\tau_3)E(\widetilde{Z}_{1,0}^s\mid X_1)\}. 
\end{align*}
Under Assumption \ref{assumption8}, we have $\text{var}(I_2)=o(N^{-2} p^{-1})$. Similar analysis shows that 
$\text{var}(I_3)=o(N^{-2} p^{-1})$, which implies that $\text{var}(\Delta_{s})=o(N^{-2} p^{-1})$.

Next, we study $\text{var}(\widetilde{\Delta}_{l})$. Following similar calculations of $\text{var}(\widetilde{\Delta}_{2})$ in Section \ref{secB.2} and using Lemmas \ref{lemmaa3}-\ref{lemmaa5}, we have
\begin{align*}
    \text{var}(\widetilde{\Delta}_{l})&=O(N^{-1} p^{-2l})\times[E\{(\|X_{1}-X_{2}\|_{2}^{2}-p\tau)^{2l}\}+E\{(\|Y_{1}-Y_{2}\|_{2}^{2}-p\tau)^{2l}\}\\
    &\quad+E\{(\|X_{1}-Y_{1}\|_{2}^{2}-p\tau)^{2l}\}]\\
    &=O(N^{-1} p^{-2l})\times[O(\text{tr}^{l}(\Sigma_{1}^2))+O(\text{tr}^{l}(\Sigma_{2}^2))+O(\text{tr}^{l}(\Sigma_{1}\Sigma_{2}))\\
    &\quad+O(\{(\mu_{1}-\mu_{2})^{\top}\Sigma_{1}(\mu_{1}-\mu_{2})\}^{l})+O(\{(\mu_{2}-\mu_{1})^{\top}\Sigma_{2}(\mu_{2}-\mu_{1})\}^{l})]\\
    &=O(N^{-1} p^{-l}).
\end{align*}

\subsection{The means of $\Delta_{s}$ for $s=2,\ldots,l-1$ when $X$ and $Y$ share the first $(l-1)$ moments}\label{secB.5}
We will show that $E(\Delta_s)=T_{s-1}=0$ for $s=2,\ldots,l-1$ when the first $(l-1)$ moments of $X$ and $Y$ match. Observe that $T_{s-1} $ defined in (\ref{eq-Ts}) involves the first $2s$ moments of $X$ and $Y$. Thus it is a non-trivial task to show that $T_{s-1}=0$ when we only assume the first $(l-1)$ moments of $X$ and $Y$ match. Using the binomial expansion formula, we have
\begin{align*}
&\quad E\{(\|X_{1}-X_{2}\|_{2}^{2}-p\tau)^{s}\}+E\{(\|Y_{1}-Y_{2}\|_{2}^{2}-p\tau)^{s}\}-2E\{(\|X_{1}-Y_{1}\|_{2}^{2}-p\tau)^{s}\}
\\&=\sum^{s}_{a=0}\binom{s}{a}(-p\tau)^{s-a}\left(E\|X_1- X_2\|^{2a}_2+E\|Y_1-Y_2\|^{2a}_2-2E\|X_1- Y_1\|^{2a}_2\right).
\end{align*}
Next we show that $E\|X_1- X_2\|^{2a}_2+E\|Y_1-Y_2\|^{2a}_2-2E\|X_1- Y_1\|^{2a}_2=0$, for $a\le s$, when $X$ and $Y$ share the first $(l-1)$ moments. By the multinomial theorem,
\begin{align*}
    E\|X_{1}-X_{2}\|_{2}^{2a}&=\sum_{a_1+a_2\le a}\frac{(-2)^{a-a_1-a_2}a!}{a_1!a_2!(a-a_1-a_2)!}E\{\|X_1\|_{2}^{2a_1}\|X_2\|_{2}^{2a_2}(X_1^{\top}X_2)^{a-a_1-a_2}\}\\
    &=\sum_{2a_1\le a}\frac{(-2)^{a-2a_1}a!}{a_1!a_1!(a-2a_1)!}E\{\|X_1\|_{2}^{2a_1}\|X_2\|_{2}^{2a_1}(X_1^{\top}X_2)^{a-2a_1}\}\\
    &\quad+2\sum_{\substack{a_1>a_2 \\ a_1+a_2\le a}}\frac{(-2)^{a-a_1-a_2}a!}{a_1!a_2!(a-a_1-a_2)!}E\{\|X_1\|_{2}^{2a_1}\|X_2\|_{2}^{2a_2}(X_1^{\top}X_2)^{a-a_1-a_2}\}\\
    &:=S_{1}+2S_{2}.
\end{align*}
Observe that $S_{1}$ merely depends on the first $a$ moments of $X$. Thus
\begin{align*}
    S_{1}&=\sum_{2a_1\le a}\frac{(-2)^{a-2a_1}a!}{a_1!a_1!(a-2a_1)!}E\{\|X_1\|_{2}^{2a_1}\|Y_1\|_{2}^{2a_1}(X_1^{\top}Y_1)^{a-2a_1}\}\\
    &=\sum_{2a_1\le a}\frac{(-2)^{a-2a_1}a!}{a_1!a_1!(a-2a_1)!}E\{\|Y_1\|_{2}^{2a_1}\|Y_2\|_{2}^{2a_1}(Y_1^{\top}Y_2)^{a-2a_1}\},
\end{align*}
when the first $(l-1)$ moments of $X$ and $Y$ match. 

Moreover, by conditioning on $X_{1}$, $E\{\|X_1\|_{2}^{2a_1}\|X_2\|_{2}^{2a_2}(X_1^{\top}X_2)^{a-a_1-a_2}\mid X_1\}$ entirely depends on the first $(a-1)$ moments of $X$ (since $2a_2+a-a_1-a_2=a+a_2-a_1<a$). Thus, 
$$E\{\|X_1\|_{2}^{2a_1}\|X_2\|_{2}^{2a_2}(X_1^{\top}X_2)^{a-a_1-a_2}\mid X_1\}=E\{\|X_1\|_{2}^{2a_1}\|Y_1\|_{2}^{2a_2}(X_1^{\top}Y_1)^{a-a_1-a_2}\mid X_1\},$$
which implies
$$S_2=\sum_{\substack{a_1>a_2 \\ a_1+a_2\le a}}\frac{(-2)^{a-a_1-a_2}a!}{a_1!a_2!(a-a_1-a_2)!}E\{\|X_1\|_{2}^{2a_1}\|Y_1\|_{2}^{2a_2}(X_1^{\top}Y_1)^{a-a_1-a_2}\}.$$
Similarly, by conditioning on $Y_1$, we can show that
\begin{align*}
    &\quad\sum_{\substack{a_1>a_2 \\ a_1+a_2\le a}}\frac{(-2)^{a-a_1-a_2}a!}{a_1!a_2!(a-a_1-a_2)!}E\{\|Y_1\|_{2}^{2a_1}\|Y_2\|_{2}^{2a_2}(Y_1^{\top}Y_2)^{a-a_1-a_2}\}\\
    &=\sum_{\substack{a_1>a_2 \\ a_1+a_2\le a}}\frac{(-2)^{a-a_1-a_2}a!}{a_1!a_2!(a-a_1-a_2)!}E\{\|Y_1\|_{2}^{2a_1}\|X_1\|_{2}^{2a_2}(Y_1^{\top}X_1)^{a-a_1-a_2}\}.
\end{align*}
Combining these results, it is straightforward to see that
$$E\|X_1- X_2\|^{2a}_2+E\|Y_1-Y_2\|^{2a}_2-2E\|X_1- Y_1\|^{2a}_2=0,\ a\le s\implies E(\Delta_{s})=T_{s-1}=0.$$

\section{Proofs of lemmas, theorems and propositions}\label{secC}
\begin{proof}[Proof of Lemma \ref{lemma3.1}]
Following similar calculations in Section 6.2 of \cite{chen2010two}, it is straightforward to show that
$$E(\Delta_{1,1})=E(\Delta_{1,2})=E(\Delta_{1,3})=0,$$
and
\begin{align*}
    \text{var}(\Delta_{1,1})&=\frac{8}{p^{2}}\bigg[\frac{1}{n(n-1)}\{f^{(1)}(\tau_{1})\}^{2}\text{tr}(\Sigma_{1}^{2})+\frac{1}{m(m-1)}\{f^{(1)}(\tau_{2})\}^{2}\text{tr}(\Sigma_{2}^{2})+\frac{2}{nm}\{f^{(1)}(\tau_{3})\}^{2}\text{tr}(\Sigma_{1}\Sigma_{2})\bigg],\\
    \text{var}(\Delta_{1,2})&=\frac{16}{p^{2}}\{f^{(1)}(\tau_{3})\}^{2}\bigg\{\frac{1}{n}(\mu_{1}-\mu_{2})^\top \Sigma_{1}(\mu_{1}-\mu_{2})+\frac{1}{m}(\mu_{2}-\mu_{1})^\top \Sigma_{2}(\mu_{2}-\mu_{1})\bigg\},\\
    \text{var}(\Delta_{1,3})&=\frac{4}{p^{2}}\bigg[\{f^{(1)}(\tau_{1})-f^{(1)}(\tau_{3})\}^{2}\times\frac{1}{n}\text{var}(U_{1}^\top\Gamma_{1}^{\top}\Gamma_{1}U_{1})+\{f^{(1)}(\tau_{2})-f^{(1)}(\tau_{3})\}^{2}\times\frac{1}{m}\text{var}(V_{1}^\top\Gamma_{2}^{\top}\Gamma_{2}V_{1})\bigg]\\
    &=\frac{4}{p^{2}}\bigg[\{f^{(1)}(\tau_{1})-f^{(1)}(\tau_{3})\}^{2}\times\frac{1}{n}\bigg\{2\text{tr}(\Sigma_{1}^{2})+\sum_{k}(E(U_{1}(k)^{4})-3)s_{kk}^{2}\bigg\}\\
    &\quad+\{f^{(1)}(\tau_{2})-f^{(1)}(\tau_{3})\}^{2}\times\frac{1}{m}\bigg\{2\text{tr}(\Sigma_{1}^{2})+\sum_{k}(E(V_{1}(k)^{4})-3)t_{kk}^{2}\bigg\}\bigg],
\end{align*}
where $\Gamma_{1}^{\top}\Gamma_{1}=(s_{kl})_{q\times q}$, $\Gamma_{2}^{\top}\Gamma_{2}=(t_{kl})_{q\times q}$, and we have used Lemma \ref{lemmaa2}. 

Under the local alternative, i.e., Assumption \ref{assumption6}, as $\text{var}(\Delta_{1,2})=o(\text{var}(\Delta_{1,1}))$ and $\text{var}(\Delta_{1,3})=o(\text{var}(\Delta_{1,1}))$, 
$$\frac{\Delta_{1}}{\surd{\text{var}(\Delta_{1})}}=\frac{\Delta_{1,1}}{\surd{\text{var}(\Delta_{1,1})}}+o_{p}(1).$$
Note that Assumptions \ref{assumption1} and \ref{assumption4} correspond to (3.1), (3.2) and (3.3) in \cite{chen2010two}. Besides, under Assumption \ref{assumption2}, one may check that
$$\text{tr}(\Sigma_{i}\Sigma_{j}\Sigma_{l}\Sigma_{h})=O(p)=o(\text{tr}^{2}[(\Sigma_{1}+\Sigma_{2})^{2}]),$$
for $i,j,l,h=$1 or 2, which is (3.6) in \cite{chen2010two}. Therefore, the asymptotic normality of $\Delta_{1,1}$ can be attained by simply mimicking the arguments in Section 6.2 in \cite{chen2010two}. The only difference is the definition of $\phi_{ij}$ in their proof. In our case,
$$\phi_{ij}=\left\{
\begin{aligned}
    n^{-1}(n-1)^{-1}f^{(1)}(\tau_1)Z_{i}^{\top}Z_{j},\quad&\text{if }i,j\in\{1,\ldots,n\},\\
    -n^{-1}m^{-1}f^{(1)}(\tau_3)Z_{i}^{\top}Z_{j},\quad&\text{if }i\in\{1,\ldots,n\}\text{ and }j\in\{n+1,\ldots,n+m\},\\
    m^{-1}(m-1)^{-1}f^{(1)}(\tau_2)Z_{i}^{\top}Z_{j},\quad&\text{if }i,j\in\{n+1,\ldots,n+m\},
\end{aligned}
\right.
$$
where $Z_{i}=X_{i}-\mu_1$ for $i=1,\ldots,n$ and $Z_{j+n}=Y_{j}-\mu_2$ for $j=1,\ldots,m$. Then the asymptotic normality can be proved with verification of the conditions of the martingale CLT (see, for example, Corollary 3.1 of \cite{hall2014martingale}). We omit the details here to avoid duplication. 

Under Assumption \ref{assumption7}, we have either $\text{var}(\Delta_{1,1})=o(\text{var}(\Delta_{1,2}))$ or $\text{var}(\Delta_{1,1})=o(\text{var}(\Delta_{1,3}))$. Therefore, 
$$\frac{\Delta_{1}}{\surd{\text{var}(\Delta_{1})}}=\frac{\Delta_{1,2}+\Delta_{1,3}}{\surd{\text{var}(\Delta_{1,2}+\Delta_{1,3})}}+o_{p}(1).$$
As $\Delta_{1,2}+\Delta_{1,3}$ is a summation of averages of independent variables, the asymptotic normality follows {\color{red}by verifying Lyapunov's condition}. 
\end{proof}

\begin{proof}[Proof of Theorem \ref{theorem3.2}]
Note $\text{var}(\Delta_{1})=\text{var}(\Delta_{1,1})\{1+o(1)\}=\Theta(N^{-2} p^{-1})$ under Assumptions \ref{assumption1}--\ref{assumption4} and \ref{assumption6}. If $N=o(p)$, 
\begin{align*}
   \text{var}(\widetilde{\Delta}_{2})&=O(N^{-1} p^{-2})=o(\text{var}({\color{red}\Delta_{1,1}})),\\ E(\widetilde{\Delta}_{2})-T_{1}&=O(p^{-\frac{3}{2}})=o(\surd{\text{var}({\color{red}\Delta_{1,1}})}), 
\end{align*}
Hence, 
$$E\{(\widetilde{\Delta}_{2}-T_{1})^{2}\}=o(\text{var}({\color{red}\Delta_{1,1}})).$$
Therefore, 
$$\frac{\textup{MMD}_{n,m}^{2}-\Delta_{0}-T_{1}}{\surd{\textup{var}({\color{red}\Delta_{1,1}})}}=\frac{\Delta_{1}}{\surd{\textup{var}({\color{red}\Delta_{1,1}})}}+\frac{\widetilde{\Delta}_{2}-T_{1}}{\surd{\textup{var}({\color{red}\Delta_{1,1}})}}=\frac{{\color{red}\Delta_{1,1}}}{\surd{\textup{var}({\color{red}\Delta_{1,1}})}}+o_{p}(1)\xrightarrow{d}N(0,1).$$
\end{proof}

\begin{proof}[Proof of Theorem \ref{theorem3.4}]
Under Assumptions \ref{assumption1}--\ref{assumption4} and \ref{assumption7}, note $\text{var}(\Delta_{1})=\text{var}(\Delta_{1,2}+\Delta_{1,3})\{1+o(1)\}=\omega(\max\{p^{-3},N^{-2} p^{-1}\})$. As in Theorem \ref{theorem3.2}, it can be checked that
\begin{align*}
   \text{var}(\widetilde{\Delta}_{2})&=O(N^{-1} p^{-2})=o(\text{var}({\color{red}\Delta_{1,2}+\Delta_{1,3}})),\\ E(\widetilde{\Delta}_{2})-T_{1}&=O(p^{-\frac{3}{2}})=o(\surd{\text{var}({\color{red}\Delta_{1,2}+\Delta_{1,3}})}).
\end{align*}
Hence, the asymptotic normality of $\text{MMD}_{n,m}^{2}$ can be justified in the same way. 
\end{proof}

\begin{proof}[Proof of Corollary \ref{corollary3.3} \& \ref{corollary3.5}]
{\color{red}Under $N=o(\surd{p})$, Corollary \ref{corollary3.3} follows as $T_{1}=O(p^{-1})=o(\surd{\text{var}(\Delta_{1,1}}))$. Corollary \ref{corollary3.5} can be justified in the same way.}
\end{proof}

\begin{proof}[Proof of Theorem \ref{theorem3.6}]
Note $\widetilde{\Delta}_{2}/\surd{\text{var}({\color{red}\Delta_{1,1}})}=o_{p}(1)$ under the null. Thus, 
$$\frac{\textup{MMD}_{n,m}^{2}}{\surd{\textup{var}({\color{red}\Delta_{1,1}})}}=\frac{{\color{red}\Delta_{1,1}}}{\surd{\textup{var}({\color{red}\Delta_{1,1}})}}+\frac{\widetilde{\Delta}_{2}}{\surd{\textup{var}({\color{red}\Delta_{1,1}})}}\xrightarrow{d}N(0,1).$$
\end{proof}

\begin{proof}[Proof of Theorem \ref{theorem3.7}]
According to Theorem 2 in \cite{chen2010two}, $\widehat{\text{tr}(\Sigma_{i}^{2})}$ and $\widehat{\text{tr}(\Sigma_{1}\Sigma_{2})}$ are ratio-consistent estimators of $\text{tr}(\Sigma_{i}^{2})$ and $\text{tr}(\Sigma_{1}\Sigma_{2})$. Thus, we only need to prove that $f^{(1)}(\widehat{\tau_{i}})$ is ratio-consistent to $f^{(1)}(\tau_{i})$. We shall show the ratio consistency of $f^{(1)}(\widehat{\tau_{1}})$ here. The proofs for $f^{(1)}(\widehat{\tau_{i}}),i=2,3$ are similar. 

We note that
\begin{align*}
\frac{1}{n}\sum_{i}X_{i}^{\top}(X_{i}-\bar{X}_{-i})&=\frac{1}{n}\sum_i \|X_i\|_{2}^{2}-\frac{1}{n(n-1)}\sum_{i_1\neq i_2}X_{i_1}^\top X_{i_2}.
\end{align*}
Then
\begin{align*}
    \text{var}(\widehat{\tau_{1}})&\le 2\times\frac{4}{p^2}\times\bigg\{\text{var}\bigg(\frac{1}{n}\sum_{i}\|X_i\|_{2}^{2}\bigg)+\text{var}\bigg(\frac{1}{n(n-1)}\sum_{i_1\neq i_2}X_{i_1}^\top X_{i_2}\bigg)\bigg\},
\end{align*}
where
\begin{align*}
    \text{var}\bigg(\frac{1}{n}\sum_{i}\|X_i\|_{2}^{2}\bigg)&=\frac{1}{n}\text{var}(\|X_1\|_{2}^{2})\\
    &\le \frac{C}{n}\times[\text{var}(\|X_1-\mu_1\|_{2}^{2})+\text{var}\{\mu_1^{\top}(X_1-\mu_1)\}]\\
    &=O(p N^{-1}),
\end{align*}
and
\begin{align*}
    \text{var}\bigg(\frac{1}{n(n-1)}\sum_{i_1\neq i_2}X_{i_1}^\top X_{i_2}\bigg)=\frac{2}{n(n-1)}\text{tr}(\Sigma_{1}^{2})+\frac{4}{n}\mu_{1}^{\top}\Sigma_{1}\mu_{1}=O(p N^{-1}). 
\end{align*}
Thus, $\text{var}(\widehat{\tau_{1}})=O(p^{-1} N^{-1})=o(1)$, which implies $\widehat{\tau_{1}}\xrightarrow{p}\tau_{1}$. By the continuous mapping theorem, we have
$$|f^{(1)}(\widehat{\tau_{1}})-f^{(1)}(\tau_{1})|=o_{p}(1)=o_{p}(f^{(1)}(\tau_{1})).$$
\end{proof}

\begin{proof}[Proof of Proposition \ref{prop1}]
When $s=1$, one can verify that $I_1=\Delta_{1,1}$ and $I_2+I_3=\Delta_{1,2}+\Delta_{1,3}$ in the decomposition in Section \ref{secB.4}. Therefore, Assumption \ref{assumption8} means that $\text{var}(\Delta_{1,2}+\Delta_{1,3})=\text{var}(I_2+I_3)=o(N^{-2}p^{-1})$. The equivalence follows. 
\end{proof}

\begin{proof}[Proof of Proposition \ref{prop2}]
Let $\Gamma_1=\Gamma_2=\Gamma=(\gamma_{kl})_{p\times q}$ and $\Gamma^\top \Gamma=(s_{kl})_{q\times q}$. Without loss of generality, we assume that $\mu=0$ (as otherwise we can work with $X_i-\mu$ and $Y_j-\mu$). Under $\mu_1=\mu_2$ and $\Sigma_1=\Sigma_2$, which implies $\tau_1=\tau_3:=\tau$, 
\begin{align*}
&\quad f^{(2)}(\tau_1)E(\widetilde{X}_{1,2}^2\mid X_1)-f^{(2)}(\tau_3)E(\widetilde{Z}_{1,1}^2\mid X_1)\\
&=p^{-2}f^{(2)}(\tau)\left\{E(\|X_1-X_2\|_2^4\mid X_1)-E(\|X_1-Y_1\|_2^4\mid X_1)\right\}
\\&=-p^{-2}f^{(2)}(\tau)\big\{4E(\|X_2\|_2^2X_2^\top-\|Y_1\|_2^2Y_1^\top)X_1+E(\|X_2\|_{2}^{4}-\|Y_1\|_{2}^{4})\big\}.
\end{align*}
Then under Assumptions \ref{assumption1}--\ref{assumption2}, we have
\begin{align*}
&\quad\text{var}\{f^{(2)}(\tau_1)E(\widetilde{X}_{1,2}^2\mid X_1)-f^{(2)}(\tau_3)E(\widetilde{Z}_{1,1}^s\mid X_1)\}\\
&\le O(p^{-4})\|\Sigma_1\|_{\rm op}\times\|E(\|X\|_2^2 X - \|Y\|_2^2Y)\|_2^2
\\&=O(p^{-4})\sum_i\bigg\{\sum_j E(x_j^2x_i-y_j^2y_i)\bigg\}^2
\\&=O(p^{-4})\sum_i\bigg\{\sum_j\sum_k \gamma_{j,k}^2\gamma_{i,k}(\mu_{k,3}^{(1)}-\mu_{k,3}^{(2)})\bigg\}^2
\\&=O(p^{-4})\sum_i\sum_{j,j'}\sum_{k,k'}
\gamma_{j,k}^2\gamma_{i,k}\gamma_{j',k'}^2\gamma_{i,k'}(\mu_{k,3}^{(1)}-\mu_{k,3}^{(2)})(\mu_{k',3}^{(1)}-\mu_{k',3}^{(2)})
\\&=O(p^{-4})\sum_{k,k'}s_{k,k}s_{k,k'}s_{k',k'}(\mu_{k,3}^{(1)}-\mu_{k,3}^{(2)})(\mu_{k',3}^{(1)}-\mu_{k',3}^{(2)})
\\&\leq O(p^{-4})\|\Gamma^\top \Gamma\|_{\rm op}\times\bigg[\sum_{k}\{s_{k,k}(\mu_{k,3}^{(1)}-\mu_{k,3}^{(2)})\}^2\times\sum_{k'}\{s_{k',k'}(\mu_{k',3}^{(1)}-\mu_{k',3}^{(2)})\}^2\bigg]^{1/2}
\\&= O(p^{-4})\sum_{k}\{s_{k,k}(\mu_{k,3}^{(1)}-\mu_{k,3}^{(2)})\}^2,
\end{align*}
where we have used the definition of the operator norm that $\|A\|_{\text{op}}=\sup_{\|a\|_2\leq 1,\|b\|_2\leq 1}a^\top A b$ in the second to the last line. 
\end{proof}

\begin{proof}[Proof of Theorem \ref{theorem3.9}]
Recall $\text{MMD}^2(P_X,P_Y)=E(\textup{MMD}_{n,m}^{2})=\Delta_0+\sum^{l-1}_{s=2}E(\Delta_s)+E(\widetilde{\Delta}_l)$. Under Assumption \ref{assumption8}, we have
$\text{var}(\Delta_s)=o(N^{-2} p^{-1})$ for $s=2,\ldots,l-1$. Moreover, when 
$N=o(p^{l-1})$, $\text{var}(\widetilde{\Delta}_l)=O(N^{-1} p^{-l})=o(N^{-2} p^{-1}).$ Therefore, we have
\begin{align*}
\frac{\textup{MMD}_{n,m}^{2}-\text{MMD}^2(P_X,P_Y)}{\surd{\textup{var}({\color{red}\Delta_{1,1}})}}&=\frac{\Delta_{1}}{\surd{\textup{var}({\color{red}\Delta_{1,1}})}}+\frac{\sum_{s=2}^{l-1}(\Delta_{s}-E(\Delta_s))}{\surd{\textup{var}({\color{red}\Delta_{1,1}})}}+\frac{\widetilde{\Delta}_{l}-E(\widetilde{\Delta}_l)}{\surd{\textup{var}({\color{red}\Delta_{1,1}})}}
\\&=\frac{{\color{red}\Delta_{1,1}}}{\surd{\textup{var}({\color{red}\Delta_{1,1}})}}+o_{p}(1)+o_{p}(1)\xrightarrow{d}N(0,1).    
\end{align*}
\end{proof}

\begin{proof}[Proof of Lemma \ref{lemma3.8}]
We divide the proof into four parts.

\textbf{Part 1.} Recall that
\begin{align*}
T_{l-1}=\frac{f^{(l)}(\tau)}{l!p^{l}}\times\left[E\{(\|X_{1}-X_{2}\|_{2}^{2}-p\tau)^{l}\}+E\{(\|Y_{1}-Y_{2}\|_{2}^{2}-p\tau)^{l}\}-2E\{(\|X_{1}-Y_{1}\|_{2}^{2}-p\tau)^{l}\}\right].
\end{align*}
Using the binomial expansion formula, we have
\begin{align*}
&\quad E\{(\|X_{1}-X_{2}\|_{2}^{2}-p\tau)^{l}\}+E\{(\|Y_{1}-Y_{2}\|_{2}^{2}-p\tau)^{l}\}-2E\{(\|X_{1}-Y_{1}\|_{2}^{2}-p\tau)^{l}\}
\\&=\sum^{l}_{a=0}\binom{l}{a}(-p\tau)^{l-a}\left(E\|X_1- X_2\|^{2a}_2+E\|Y_1-Y_2\|^{2a}_2-2E\|X_1- Y_1\|^{2a}_2\right).
\end{align*}
Following the arguments in Section \ref{secB.5}, we can show that
$E\|X_1- X_2\|^{2a}_2+E\|Y_1-Y_2\|^{2a}_2-2E\|X_1- Y_1\|^{2a}_2=0$ whenever $a\le l-1.$ Thus, we only need to consider the term with $a=l$ in the summation. Without loss of generality, we set $\mu=0$ in the arguments below. Note that
\begin{align*}
&\quad E\|X_1- X_2\|^{2l}_2+E\|Y_1-Y_2\|^{2l}_2-2E\|X_1- Y_1\|^{2l}_2
\\&=\sum_{0\leq a_1+a_2\leq l}\frac{(-2)^{l-a_1-a_2}l!}{a_1!a_2!(l-a_1-a_2)!}\Big[E\{\|X_1\|_{2}^{2a_1}\|X_2\|_{2}^{2a_2}(X_1^\top X_2)^{l-a_1-a_2}\}
\\&\quad+
E\{\|Y_1\|_{2}^{2a_1}\|Y_2\|_{2}^{2a_2}(Y_1^\top Y_2)^{l-a_1-a_2}\}-2E\{\|X_1\|_{2}^{2a_1}\|Y_1\|_{2}^{2a_2}(X_1^\top Y_1)^{l-a_1-a_2}\}\Big].
\end{align*}
Again using the arguments in Section \ref{secB.5}, we can show that the summands will be non-zero only when the degrees of $X_1$ and $X_2$ ($Y_1$ and $Y_2$, $X_1$ and $Y_1$) are equal, i.e., $a_1=a_2$. When $a_1=a_2$, the summand inside the square brackets becomes
\begin{align*}
\sum_{1\leq i_1,\ldots,i_{l-a_1-a_2}\leq p}(\mu_{i_1,\ldots,i_{l-a_1-a_2}}^{(1)}-\mu_{i_1,\ldots,i_{l-a_1-a_2}}^{(2)})^2=\|\mathcal{T}_{1,l}^{(a_1,a_2)}-\mathcal{T}_{2,l}^{(a_1,a_2)}\|_{\rm F}^2,
\end{align*}

which thus implies the desired result.

Let $\Gamma_1=\Gamma_2=\Gamma=(\gamma_{kl})_{p\times q}$, $\Gamma^\top \Gamma=(s_{kl})_{q\times q}$, $\mu_{k,s}^{(1)}=E(U_1(k)^s)$ and $\mu_{k,s}^{(2)}=E(V_1(k)^s)$ for $k=1,\ldots,q$ and $s\ge 1$. Note that $\max_{k,l}|s_{kl}|\leq K$ under Assumption \ref{assumption2}. For $a_1=a_2=0$, as $\mu_{k,s}^{(1)}=\mu_{k,s}^{(2)}$ for $s=1,\ldots,l-1$, we have
\begin{align*}
\|\mathcal{T}_{1,l}^{(0,0)}-\mathcal{T}_{2,l}^{(0,0)}\|_{\rm F}^2&=\sum_{i_1,\ldots,i_l}\bigg\{\sum_k\gamma_{i_1,k}\gamma_{i_2,k}\cdots\gamma_{i_l,k}(\mu_{k,l}^{(1)}-\mu_{k,l}^{(2)})\bigg\}^2
\\&=\sum_{k_1,k_2}s_{k_1,k_2}^l(\mu_{k_1,l}^{(1)}-\mu_{k_1,l}^{(2)})(\mu_{k_2,l}^{(1)}-\mu_{k_2,l}^{(2)})
\\&\leq CK^{l-2}\sum_{k_1,k_2}s_{k_1,k_2}^2=CK^{l-2}\text{tr}(\Gamma^\top\Gamma\Gamma^\top\Gamma)=CK^{l-2}\text{tr}(\Sigma^2)=O(p).
\end{align*}
Now consider $a_1=a_2\geq 1$. Based on similar calculation as above, we obtain
\begin{align*}
\|\mathcal{T}_{1,l}^{(a_1,a_2)}-\mathcal{T}_{2,l}^{(a_1,a_2)}\|_{\rm F}^2 
=\sum_{k_1,k_2}s_{k_1,k_2}^{l-2a_1}s_{k_1,k_1}^{a_1}s_{k_2,k_2}^{a_1}(\mu_{k_1,l}^{(1)}-\mu_{k_1,l}^{(2)})(\mu_{k_2,l}^{(1)}-\mu_{k_2,l}^{(2)}).
\end{align*}
When $l=2r-1$ for $r\ge 2$, we have $l-2a_1\ge 1$, and thus
\begin{align*}
\|\mathcal{T}_{1,l}^{(a_1,a_2)}-\mathcal{T}_{2,l}^{(a_1,a_2)}\|_{\rm F}^2 
&=\sum_{k_1,k_2}s_{k_1,k_2}s_{k_1,k_2}^{l-2a_1-1}s_{k_1,k_1}^{a_1}s_{k_2,k_2}^{a_1}(\mu_{k_1,l}^{(1)}-\mu_{k_1,l}^{(2)})(\mu_{k_2,l}^{(1)}-\mu_{k_2,l}^{(2)})\\
&\le CK^{l-1}\sum_{k_1,k_2}s_{k_1,k_2}=O(p).
\end{align*}
When $l=2r$ for $r\ge 2$, $\|\mathcal{T}_{1,l}^{(a_1,a_2)}-\mathcal{T}_{2,l}^{(a_1,a_2)}\|_{\rm F}^2=O(p)$ when $l-2a_1>0$ (i.e., $a_1=a_2<r$) as argued above. However, when $l=2r$ and $l-2a_1=0$, i.e., $a_1=a_2=r$, we have
\begin{align*}
\|\mathcal{T}_{1,l}^{(a_1,a_2)}-\mathcal{T}_{2,l}^{(a_1,a_2)}\|_{\rm F}^2 
&=\sum_{k_1,k_2}s_{k_1,k_1}^{r}s_{k_2,k_2}^{r}(\mu_{k_1,l}^{(1)}-\mu_{k_1,l}^{(2)})(\mu_{k_2,l}^{(1)}-\mu_{k_2,l}^{(2)})\\
&=\bigg\{\sum_{k}s_{k,k}^{r}(\mu_{k,l}^{(1)}-\mu_{k,l}^{(2)})\bigg\}^{2}=O(p^2).
\end{align*}
Thus, $T_{l-1}=O(p^{-2r+2})$ when $l=2r-1$ or $l=2r$.

\textbf{Part 2.} Next we study $T_{s-1}$ for $s\geq l+1$. Note that
\begin{align*}
&\quad E\{(\|X_{1}-X_{2}\|_{2}^{2}-p\tau)^{s}\}+E\{(\|Y_{1}-Y_{2}\|_{2}^{2}-p\tau)^{s}\}-2E\{(\|X_{1}-Y_{1}\|_{2}^{2}-p\tau)^{s}\}
\\&=\sum^{s}_{a=0}\binom{s}{a}(-p\tau)^{s-a}\left(E\|X_1- X_2\|^{2a}_2+E\|Y_1-Y_2\|^{2a}_2-2E\|X_1- Y_1\|^{2a}_2\right)
\\&=\sum^{s}_{a=l}\binom{s}{a}(-p\tau)^{s-a}\left(E\|X_1- X_2\|^{2a}_2+E\|Y_1-Y_2\|^{2a}_2-2E\|X_1- Y_1\|^{2a}_2\right),
\end{align*}
where the last equality follows from the arguments in Section \ref{secB.5}. 
Without loss of generality, we set $\mu=0$ in the argument below. Note that
\begin{align*}
&\quad E\|X_1- X_2\|^{2a}_2+E\|Y_1-Y_2\|^{2a}_2-2E\|X_1- Y_1\|^{2a}_2
\\&=\sum_{0\leq a_1+a_2\leq a}\frac{(-2)^{a-a_1-a_2}a!}{a_1!a_2!(a-a_1-a_2)!}\Big[E\{\|X_1\|_{2}^{2a_1}\|X_2\|_{2}^{2a_2}(X_1^\top X_2)^{a-a_1-a_2}\}
\\&\quad+E\{\|Y_1\|_{2}^{2a_1}\|Y_2\|_{2}^{2a_2}(Y_1^\top Y_2)^{a-a_1-a_2}\}-E\{\|X_1\|_{2}^{2a_1}\|Y_1\|_{2}^{2a_2}(X_1^\top Y_1)^{a-a_1-a_2}\}
\\&\quad-E\{\|X_1\|_{2}^{2a_2}\|Y_1\|_{2}^{2a_1}(X_1^\top Y_1)^{a-a_1-a_2}\}\Big].
\end{align*}
Using the arguments in Section \ref{secB.5}, we can show that the summand vanishes whenever $a+a_1-a_2<l$ or $a+a_2-a_1<l.$
Therefore, we only need to consider those terms with $a-l\geq |a_1-a_2|$. Expanding the terms inside the square brackets, we obtain
\begin{align*}
&\quad E\|X_1- X_2\|^{2a}_2+E\|Y_1-Y_2\|^{2a}_2-2E\|X_1- Y_1\|^{2a}_2
\\&=\sum_{0\leq a_1+a_2\leq a,|a_1-a_2|\leq a-l}\frac{(-2)^{a-a_1-a_2}a!}{a_1!a_2!(a-a_1-a_2)!}\sum_{i_1,\ldots,i_{a-a_1-a_2}}\sum_{j_1,\ldots,j_{a_1}}\sum_{j_1',\ldots,j_{a_2}'}
\\&\quad\Bigg\{E\left(\prod^{a-a_1-a_2}_{s=1}x_{i_s}\prod^{a_1}_{s=1}x_{j_s}^2\right)E\left(\prod^{a-a_1-a_2}_{s=1}x_{i_s}\prod^{a_2}_{s=1}x_{j_s'}^2\right)
+E\left(\prod^{a-a_1-a_2}_{s=1}y_{i_s}\prod^{a_1}_{s=1}y_{j_s}^2\right)E\left(\prod^{a-a_1-a_2}_{s=1}y_{i_s}\prod^{a_2}_{s=1}y_{j_s'}^2\right)
\\&\quad-E\left(\prod^{a-a_1-a_2}_{s=1}x_{i_s}\prod^{a_1}_{s=1}x_{j_s}^2\right)E\left(\prod^{a-a_1-a_2}_{s=1}y_{i_s}\prod^{a_2}_{s=1}y_{j_s'}^2\right)
-E\left(\prod^{a-a_1-a_2}_{s=1}x_{i_s}\prod^{a_2}_{s=1}x_{j_s'}^2\right)E\left(\prod^{a-a_1-a_2}_{s=1}y_{i_s}\prod^{a_1}_{s=1}y_{j_s}^2\right)
\Bigg\}
\\&=\sum_{0\leq a_1+a_2\leq a,|a_1-a_2|\leq a-l}\frac{(-2)^{a-a_1-a_2}a!}{a_1!a_2!(a-a_1-a_2)!}\sum_{i_1,\ldots,i_{a-a_1-a_2}}
\\&\quad\sum_{j_1,\ldots,j_{a_1}}\Bigg\{E\left(\prod^{a-a_1-a_2}_{s=1}x_{i_s}\prod^{a_1}_{s=1}x_{j_s}^2\right)-
E\left(\prod^{a-a_1-a_2}_{s=1}y_{i_s}\prod^{a_1}_{s=1}y_{j_s}^2\right)\Bigg\}
\\&\quad\times \sum_{j_1',\ldots,j_{a_2}'}\Bigg\{E\left(\prod^{a-a_1-a_2}_{s=1}x_{i_s}\prod^{a_2}_{s=1}x_{j_s'}^2\right)-
E\left(\prod^{a-a_1-a_2}_{s=1}y_{i_s}\prod^{a_2}_{s=1}y_{j_s'}^2\right)\Bigg\}
.
\end{align*}
By the Cauchy–Schwarz inequality, we have 
\begin{align*}
&\quad E\|X_1- X_2\|^{2a}_2+E\|Y_1-Y_2\|^{2a}_2-2E\|X_1- Y_1\|^{2a}_2
\\&\leq \sum_{0\leq a_1+a_2\leq a,|a_1-a_2|\leq a-l}\frac{(-2)^{a-a_1-a_2}a!}{a_1!a_2!(a-a_1-a_2)!}
\|\mathcal{T}_{1,a}^{(a_1,a_2)}-\mathcal{T}_{2,a}^{(a_1,a_2)}\|_{\rm F}
\|\mathcal{T}_{1,a}^{(a_2,a_1)}-\mathcal{T}_{2,a}^{(a_2,a_1)}\|_{\rm F}.
\end{align*}
Under the assumption 
$\|\mathcal{T}_{1,a}^{(a_1,a_2)}-\mathcal{T}_{2,a}^{(a_1,a_2)}\|_{\rm F}^2=O(p^{a+a_1-a_2-2r+2})$
when $l=2r-1$ or $l=2r$ with $r\geq 2$, we obtain
\begin{align*}
E\|X_1- X_2\|^{2a}_2+E\|Y_1-Y_2\|^{2a}_2-2E\|X_1- Y_1\|^{2a}_2=O(p^{a-2r+2}),
\end{align*}
for $l=2r-1$ or $l=2r$ with $r\geq 2.$ It implies that 
$T_{s-1}=O(p^{-2r+2})$ when $l=2r-1$ or $l=2r$ for $l+1\leq s\leq 2(l-1)-1$.

\textbf{Part 3.} Now we verify Assumption \ref{assumption9} when $X_i=U_i+\mu$ and $Y_j=V_j+\mu$.
As $X$ and $Y$ share the first $(l-1)$ moments, for $\mathcal{T}_{1,a}^{(a_1,a_2)}-\mathcal{T}_{2,a}^{(a_1,a_2)}$ to be nonzero, $i_1,\ldots,i_{a-a_1-a_2}$ and $j_1,\ldots,j_{a_1}$ must share some common indices. Suppose $c_1$ indices out of $i_1,\ldots,i_{a-a_1-a_2}$
and $c_2$ indices out of $j_1,\ldots,j_{a_1}$ are the same, where $c_1+2c_2\geq l.$ 

We consider two cases: (i) $c_1\ge 1$; (ii) $c_1=0$. When $c_1\ge 1$, fixing the common indices and summing over all other indices in $j_1,\ldots,j_{a_1}$, the order of $\mathcal{T}_{1,a}^{(a_1,a_2)}-\mathcal{T}_{2,a}^{(a_1,a_2)}$ is at most $p^{a_1-c_2}$. Now summing over $i_1,\ldots,i_{a-a_1-a_2}$ (the number of free indices is at most $a-a_1-a_2-c_1+1$), 
we have 
\begin{align*}
\|\mathcal{T}_{1,a}^{(a_1,a_2)}-\mathcal{T}_{2,a}^{(a_1,a_2)}\|_{\rm F}^2&=O(p^{a-a_1-a_2-c_1+1+2(a_1-c_2)})=O(p^{a+a_1-a_2-c_1-2c_2+1})\\
&=O(p^{a+a_1-a_2-l+1})=O(p^{a+a_1-a_2-2r+2}),  
\end{align*}
where we have used the fact that $c_1+2c_2\geq l.$

When $c_1=0$ and $l=2r-1$, we have $2c_2>l$. Summing over $j_1,\ldots,j_{a_1}$, the order of $\mathcal{T}_{1,a}^{(a_1,a_2)}-\mathcal{T}_{2,a}^{(a_1,a_2)}$ is at most $p^{a_1-c_2+1}$. Now summing over $i_1,\ldots,i_{a-a_1-a_2}$ (the number of free indices is at most $a-a_1-a_2$ as $c_1=0$), 
we have 
\begin{align*}
\|\mathcal{T}_{1,a}^{(a_1,a_2)}-\mathcal{T}_{2,a}^{(a_1,a_2)}\|_{\rm F}^2=O(p^{a-a_1-a_2+2(a_1-c_2+1)})=O(p^{a+a_1-a_2-2c_2+2})=O(p^{a+a_1-a_2-2r+2}),  
\end{align*}
where we have used the fact that $2c_2>l$, i.e., $2c_2\ge l+1=2r$. 

Similarly, when $c_1=0$ and $l=2r$, we have $2c_2\ge l$, and
\begin{align*}
\|\mathcal{T}_{1,a}^{(a_1,a_2)}-\mathcal{T}_{2,a}^{(a_1,a_2)}\|_{\rm F}^2=O(p^{a-a_1-a_2+2(a_1-c_2+1)})=O(p^{a+a_1-a_2-2c_2+2})=O(p^{a+a_1-a_2-2r+2}).
\end{align*}

\textbf{Part 4.} In the end, applying the Taylor expansions up to order $2(l-1)$ to the sample MMD, we have \begin{align*}
\textup{MMD}_{n,m}^{2}=\Delta_{0}+\sum^{2(l-1)-1}_{s=1}\Delta_{s}+\widetilde{\Delta}_{2(l-1)}.    
\end{align*}
Under $\mu_{1}=\mu_{2}$, $\Sigma_{1}=\Sigma_{2}$, $E(U_{1}(k)^{s})=E(V_{1}(k)^{s})$ for $s\leq l-1$ and $1\leq k\leq q$, we have $T_s=0$ for $1\leq s\leq l-2$. By taking expectation on both sides, we get
$$\text{MMD}^2(P_X,P_Y)=\sum^{2(l-1)-1}_{s=l}T_{s-1}+E(\widetilde{\Delta}_{2(l-1)}).$$ 
As $E\{(\|X_1-X_2\|^2-p\tau_1)^{s}\}=O(p^{s/2})$, we can show that
$E(\widetilde{\Delta}_{2(l-1)})=O(p^{-l+1})$. 

Combining with the result that $T_{s-1}=O(p^{-l+1})$ when $l=2r-1$, and $T_{s-1}=O(p^{-l+2})$ when $l=2r$, for $l\leq s\leq 2(l-1)-1$, we obtain $\textup{MMD}^2(P_X,P_Y)=O(p^{-l+1})$ when $l=2r-1$, and $\textup{MMD}^2(P_X,P_Y)=O(p^{-l+2})$ when $l=2r$. 
\end{proof}

\begin{proof}[Proof of Corollary \ref{Coro-trivial-1}]
By Lemma \ref{lemma3.8},
$$\text{MMD}^2(P_X,P_Y)/\surd{\text{var}({\color{red}\Delta_{1,1}})}=o(1)$$ 
provided that $N=o(p^{l-3/2})$ when $l=2r-1$, or $N=o(p^{l-5/2})$ when $l=2r$. The result thus follows from Theorem \ref{theorem3.9}.
\end{proof}

\begin{proof}[Proof of Lemma \ref{lemma-delta0}]
As the ``if'' part is trivial, we focus on the ``only if'' part. Note that
\begin{align*}
0&=\frac{1}{2}\{f(\tau_1)+f(\tau_2)\}-f(\tau_3)
\geq f\left(\frac{\tau_1+\tau_2}{2}\right)-f(\tau_3)\geq 0,
\end{align*}
where the first inequality follows from the convexity and the second inequality is due to the monotonicity. Thus both inequalities become equalities and we must have $\tau_1=\tau_2$ and $\tau_1+\tau_2=2\tau_3$, which imply the desired result.
\end{proof}

\section{Additional simulations}\label{secD}
\subsection{Accuracy of the normal approximation under the null}
\begin{figure}
    \centering
    \includegraphics[width=\textwidth]{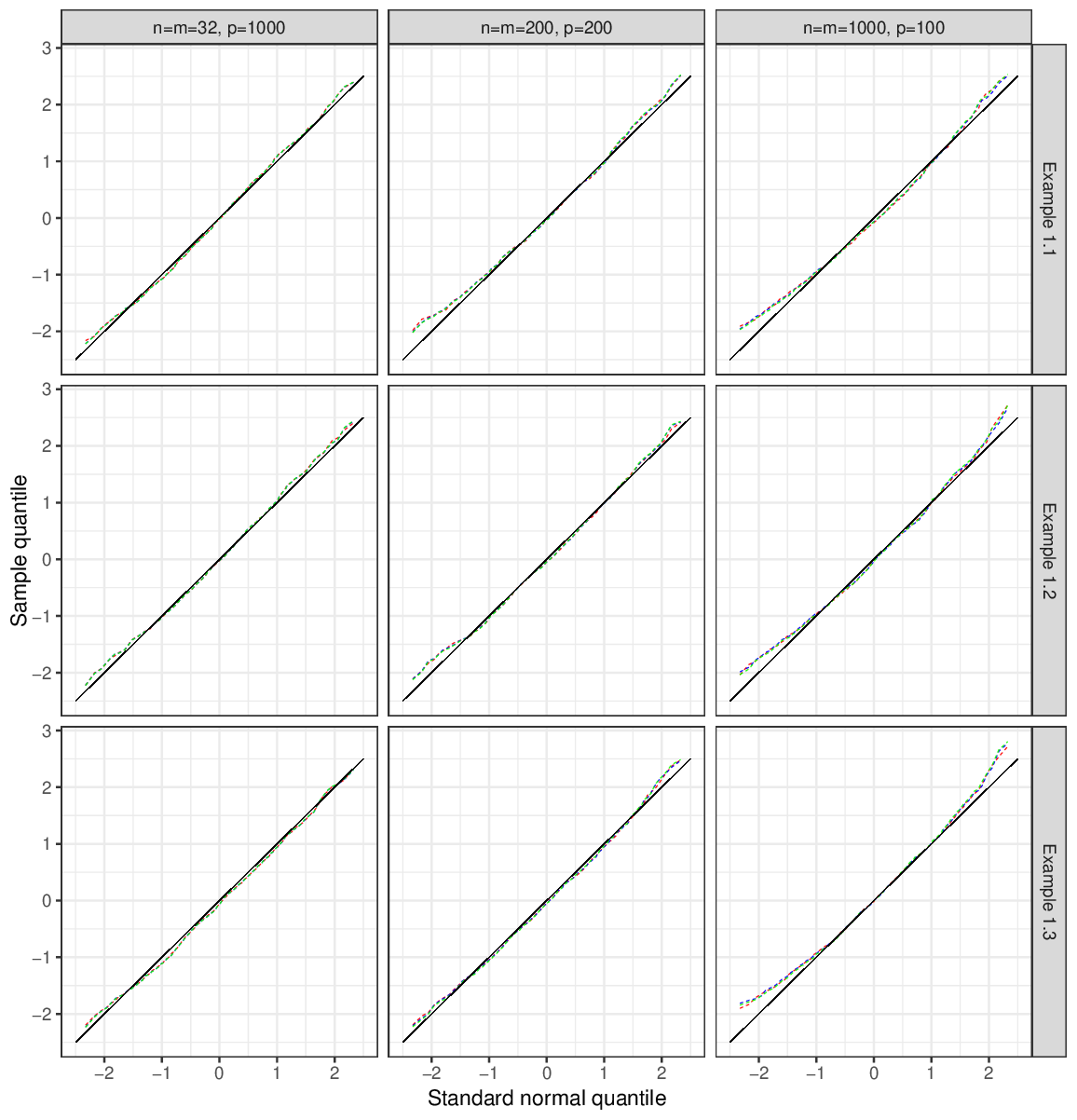}
    \caption{\footnotesize QQ plots for the MMD-based test statistics, where the results are obtained based on 1000 replications. The dashed red, dashed blue, dashed green, and solid black lines correspond to the energy distance, Gaussian kernel, Laplace kernel and the 45-degree line, respectively.}
    \label{fig:null}
\end{figure}

To verify the null CLT derived in Section \ref{sec3.3}, we first present the results of Example \ref{example1} in Figure \ref{fig:null}. As can be seen, the normal approximation appears to be quite accurate in all cases.

\subsection{How the accuracy of the normal approximation changes with respect to the dimension/sample sizes}
\begin{figure}
    \centering
    \includegraphics[width=\textwidth]{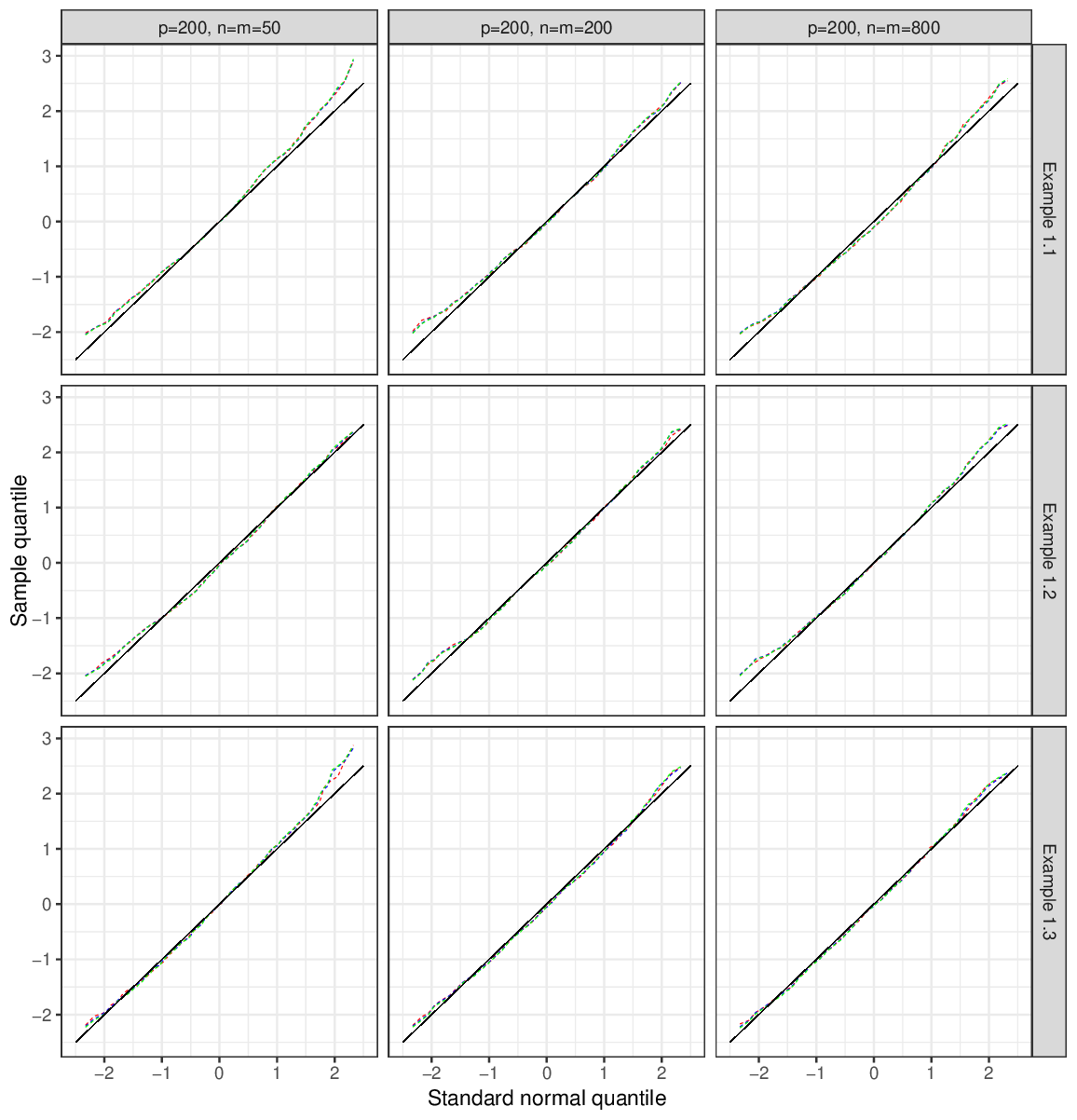}
    \caption{\footnotesize QQ plots for the MMD-based test statistics, where the dimension is fixed, the sample sizes are varying, and the results are obtained based on 1000 replications. The dashed red, dashed blue, dashed green, and solid black lines correspond to the energy distance, Gaussian kernel, Laplace kernel and the 45-degree line, respectively. }
    \label{fig:fixp}
\end{figure}

\begin{figure}
    \centering
    \includegraphics[width=\textwidth]{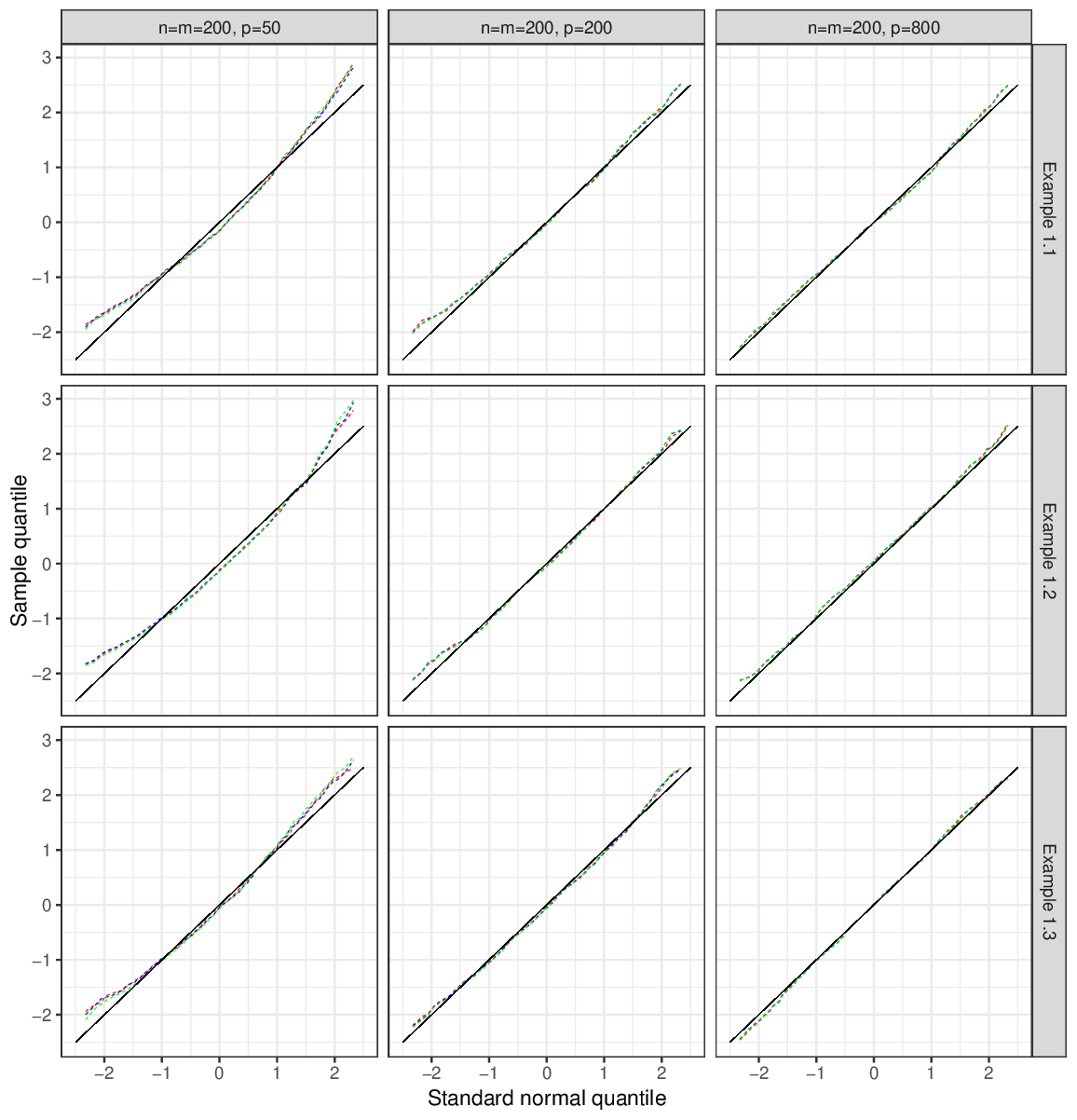}
    \caption{\footnotesize QQ plots for the MMD-based test statistics, where the sample sizes are fixed, the dimension is varying and the results are obtained based on 1000 replications. The dashed red, dashed blue, dashed green, and solid black lines correspond to the energy distance, Gaussian kernel, Laplace kernel and the 45-degree line, respectively.}
    \label{fig:fixn}
\end{figure}
To examine how the accuracy of the normal approximation is affected by the dimension and sample sizes, we follow the setting in Example \ref{example1} and consider the following two scenarios:
\begin{enumerate}
    \item Fixing the sample size $n=200$, we compare the normal approximation for different dimensions $p\in\{50,200,800\}$. 
    \item Fixing the dimension $p=200$, we compare the normal approximation for different sample sizes $n\in\{50,200,800\}$.
\end{enumerate}

In view of Figure \ref{fig:fixn}, the accuracy of the normal approximation improves when the sample size $n$ is fixed and the dimension $p$ grows, which demonstrates ``the blessing of dimensionality''. However, from Figure \ref{fig:fixp}, the accuracy of the normal approximation does not seem to improve much when the dimension $p$ is fixed and the sample size $n$ grows. Moreover, the normal approximation appears to be accurate even when $n$ is small, which can also be seen from Fig. \ref{fig:null}. 
This is not surprising as {\color{red}$\text{var}(\widetilde{\Delta}_{2})=O(N^{-2}p^{-2})$ and $\text{var}(\Delta_{1,1})=\Theta(N^{-2}p^{-1})$. It is $p\rightarrow\infty$ that ensures the remainder term $\widetilde{\Delta}_{2}$ is asymptotically negligible. }

\subsection{Comparison with other methods}
In this section, we compare the power of the MMD-based test with the approaches in \citet{chen2017new} and \citet{chakraborty2021new}. For the MMD-based test, we use the Gaussian kernel with the bandwidth $\gamma=2p$. For Chen and Friedman's graph-based test, we construct a minimal spanning tree as the similarity graph, using the $L_{2}$ distance. For Chakraborty and Zhang's method, we consider the unit group size in their new distance. We follow the settings in Examples \ref{example2}--\ref{example3} and set the significance level $\alpha=0.05$. 

\begin{figure}
    \centering
    \includegraphics[width=\textwidth]{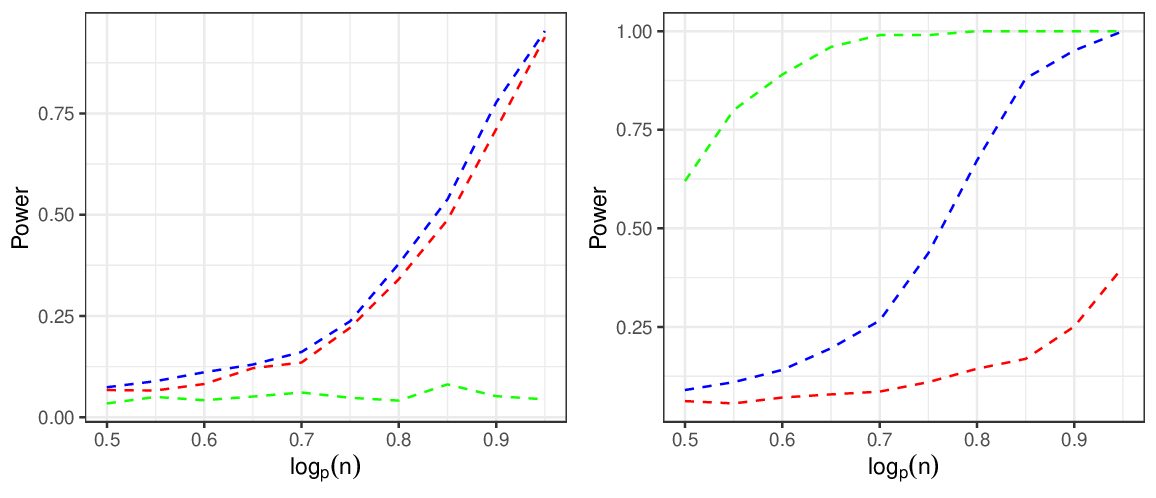}
    \caption{\footnotesize Empirical power for different tests in Example \ref{example2} (left) and Example \ref{example3} (right). The results are obtained based on 1000 replications. The dashed green, red and blue lines represent the power curves associated with Chen and Friedman's graph-based test, Chakraborty and Zhang's test, and the proposed MMD-based test, respectively. }
    \label{fig:cpr}
\end{figure}
As seen from Figure \ref{fig:cpr}, the MMD-based test and Chakraborty and Zhang's test perform similarly in Example \ref{example2}, which dominate the graph-based test. In Example \ref{example3}, the graph-based test delivers the highest power followed by the MMD-based test which dominates Chakraborty and Zhang's test.

\subsection{Data from a general distribution other than the multivariate model (\ref{eq-m})}
We consider two data generating processes: (i) the Dirichlet distribution and (ii) the uniform distribution on the Euclidean sphere in $\mathbb{R}^{p}$ with the center at the origin and the radius equal to $\surd{p}$. We intend to numerically check whether the normal approximation under the null hypothesis is still valid in these two cases. The former case is motivated by the compositional data analysis \citep{greenacre2021compositional}, while the latter case is adopted in the study of neural networks and kernel methods in the high dimension \citep{ghorbani2020neural,ghorbani2021linearized}. As in Example \ref{example1}, we consider three different high-dimensional settings $(n,p)\in\{(32,1000),(200,200),(1000,100)\}$. 
\begin{example}
$X_i/p\stackrel{i.i.d.}{\sim}\text{Dir}(1,\ldots,1)$ for $i=1,\ldots,n$, and $Y_j/p\stackrel{i.i.d.}{\sim}\text{Dir}(1,\ldots,1)$ for $j=1,\ldots,m$. The scaling factor $p$ is used to ensure that $\text{tr}(\Sigma_1)=\text{tr}(\Sigma_2)=\Theta(p)$. 
\end{example}
\begin{example}
$X_i\stackrel{i.i.d.}{\sim}\text{Uniform}(\surd{p}\mathbb{S}^{p-1})$ for $i=1,\ldots,n$, and $Y_j\stackrel{i.i.d.}{\sim}\text{Uniform}(\surd{p}\mathbb{S}^{p-1})$ for $j=1,\ldots,m$, where $\mathbb{S}^{p-1}$ denotes the unit sphere of $\mathbb{R}^p$.
\end{example}
\begin{figure}
    \centering
    \includegraphics[width=\textwidth]{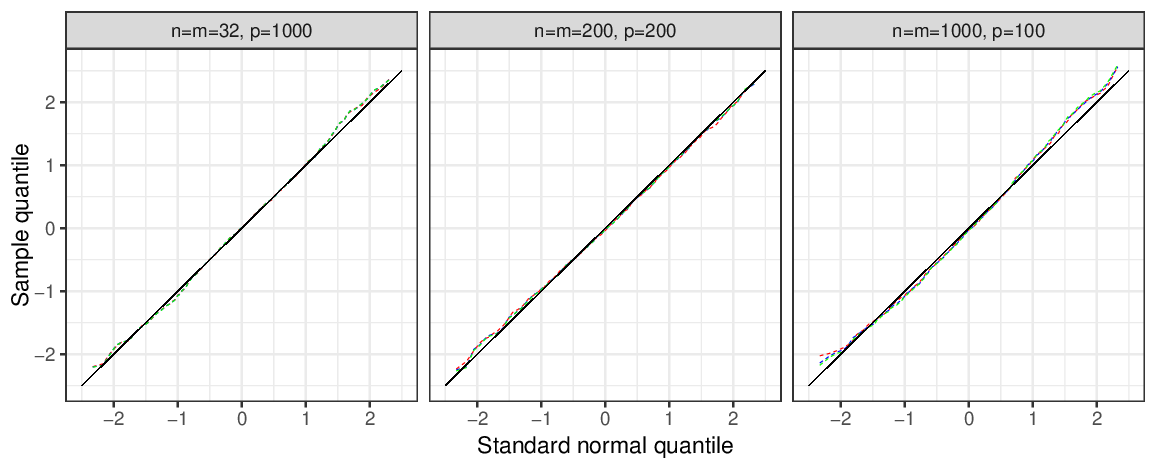}
    \caption{\footnotesize QQ plots for the MMD-based test statistics for data from a Dirichlet distribution, where the results are obtained based on 1000 replications. The dashed red, dashed blue, dashed green, and solid black lines correspond to the energy distance, Gaussian kernel, Laplace kernel and the 45-degree line, respectively.}
    \label{fig:dir}
\end{figure}

\begin{figure}
    \centering
    \includegraphics[width=\textwidth]{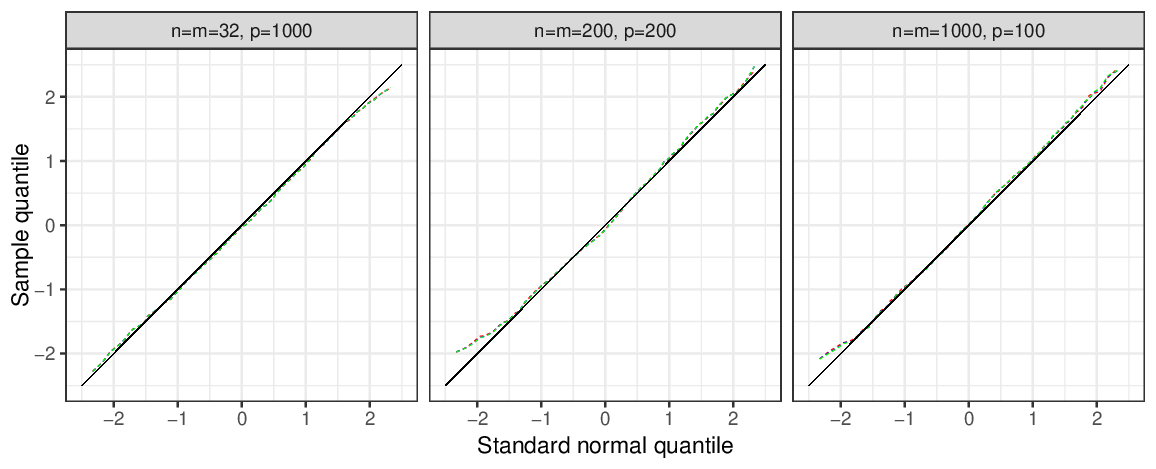}
    \caption{\footnotesize QQ plots for the MMD-based test statistics for data from a spherical distribution, where the results are obtained based on 1000 replications. The dashed red, dashed blue, dashed green, and solid black lines correspond to the energy distance, Gaussian kernel, Laplace kernel and the 45-degree line, respectively.}
    \label{fig:sphere}
\end{figure}
As seen from Figures \ref{fig:dir}--\ref{fig:sphere}, the normal approximation still appears to be accurate under the null hypothesis for these two models which do not satisfy the general multivariate model specified in (\ref{eq-m}) of the main paper.

\section{Impact of kernel and bandwidth on power}\label{secE}
Here we consider the case where $\mu_1\neq \mu_2$ and $\Sigma_1=\Sigma_2$ in the regime $N=o(p)$. \citet{ramdas2015adaptivity} found that the MMD-based tests with different kernels and bandwidths have asymptotically equal power under the mean difference alternatives. We observe a similar phenomenon in Example \ref{example2}, where the choice of kernels does not affect the asymptotic power. Indeed, the bandwidth also plays no role. To understand this phenomenon, let us assume that $\|\mu_1-\mu_2\|_2^2=\Theta(1)$ as in Example \ref{example2}.  
Then we have
\begin{align*}
\Delta_0=-2f^{(1)}(\tau_1)\frac{\|\mu_1-\mu_2\|_2^2}{\gamma}\{1+o(1)\},  
\end{align*}
by the Taylor expansion. On the other hand, we have
\begin{align*}
    \text{var}({\color{red}\Delta_{1,1}})&=\frac{8}{\gamma^2}\{f^{(1)}(\tau_1)\}^{2}\bigg\{\frac{1}{n(n-1)}\text{tr}(\Sigma_{1}^{2})+\frac{1}{m(m-1)}\text{tr}(\Sigma_{2}^{2})+\frac{2}{nm}\text{tr}(\Sigma_{1}\Sigma_{2})\bigg\}\{1+o(1)\}.
\end{align*}
By (\ref{T1-2}), it is not hard to verify that $T_1=O(p^{-2})=o(\Delta_0)$. Suppose $f^{(1)}(\tau_1)<0$ which is the case for energy distance and MMD with the Gaussian kernel or Laplace kernel. Then the asymptotic power relies on 
\begin{align*}
\frac{\Delta_0}{\surd{\text{var}({\color{red}\Delta_{1,1}})}}=\frac{\|\mu_1-\mu_2\|_2^2}{\surd{\frac{2}{n(n-1)}\text{tr}(\Sigma_{1}^{2})+\frac{2}{m(m-1)}\text{tr}(\Sigma_{2}^{2})+\frac{4}{nm}\text{tr}(\Sigma_{1}\Sigma_{2})}}\{1+o(1)\},    
\end{align*}
which is independent of the kernel and bandwidth. 

\end{appendices}

\newpage

\bibliography{main}

\end{document}